\documentclass[onecolumn, oneside, a4paper, reqno]{amsart}
\usepackage{amssymb}
\usepackage{mathrsfs}
\usepackage{anyfontsize}
\usepackage{mathtools}
\usepackage{amsmath}
\usepackage{tikz}
\usepackage{pifont}
\usepackage{color}
\usepackage{soul}
\usepackage[all]{xy}
\usepackage{type1cm}
\usepackage{hyperref}
\usepackage{color}
\usepackage{tikz-cd}
\usepackage{appendix}
\usepackage{mathbbol}
\usepackage[english]{babel}
\hypersetup{colorlinks=true,
	breaklinks=true,
	linkcolor=blue,
	pageanchor=true}
\newcommand{\xiaowuhao}{\fontsize{9pt}{\baselineskip}\selectfont}

\newcommand{\xiaoliuhao}{\fontsize{7pt}{\baselineskip}\selectfont}

\newtheorem{thm}{Theorem}[section]

\newtheorem{cor}[thm]{Corollary}
\newtheorem{lem}[thm]{Lemma}
\newtheorem{prop}[thm]{Proposition}
\newtheorem{ques}[thm]{Question}

\theoremstyle{definition}
\newtheorem{defn}[thm]{Definition}
\newtheorem{exam}[thm]{Example}
\newtheorem{rem}[thm]{Remark}

\newcommand{\ga}{\gamma}

\newcommand{\D}{\mathcal D}
\newcommand{\E}{\mathcal E}

\newcommand{\K}{\mathcal K}

\renewcommand{\S}{\mathcal S}
\newcommand{\R}{\mathcal R}
\newcommand{\T}{\mathcal T}
\newcommand{\U}{\mathcal U}
\newcommand{\V}{\mathcal V}
\newcommand{\W}{\mathcal W}

\renewcommand{\dim}{\mathsf{dim}}

\newcommand{\Holim}{\underleftarrow{{\rm Holim}}}
\newcommand{\colim}{\underrightarrow{{\rm colim}}}
\newcommand{\lims}{\underleftarrow{{\rm lim}}}

\DeclareMathOperator{\Hom}{\mathsf{Hom}}

\DeclareMathOperator*{\inj}{\!-\mathsf{inj}}

\DeclareMathOperator*{\Mod}{\!-\mathsf{Mod}}

\DeclareMathOperator*{\op}{\mathsf{op}}

\DeclareMathOperator*{\proj}{\!-\mathsf{proj}}

\DeclareMathOperator*{\smod}{\!-\mathsf{mod}}

\title[Chen-Sun-Zhang]{Triangulated categories with a compact silting object, Brown--Comenetz duality and Brown representability theorems}

\author[Chen-Sun-Zhang]{Xiaohu Chen, Yongliang Sun and Yaohua Zhang*}

\address{\normalfont{Xiaohu Chen \\
		School of Mathematical Sciences, Capital Normal University, 100048 Beijing, People's Republic of China}}
\email{xiaohu.chen@cnu.edu.cn}

\address{\normalfont{Yongliang Sun \\ 
		School of Mathematics and Physics\\
		Yancheng Institute of Technology\\
		224003 Jiangsu, People's Republic of China}}
\email{syl13536@126.com}

\address{\normalfont{Yaohua Zhang \\ Hubei Key Laboratory of Applied Mathematics, Faculty of Mathematics and Statistics, Hubei University, 430062 Wuhan, People's Republic of China}}
\email{yhzhang@hubu.edu.cn}

\thanks{* Corresponding author.}



\begin{document}
	\begin{abstract}
The paper develops a Brown--Comenetz dual framework for Neeman's representability theorems for triangulated categories with a single compact generator (Invent. math., 244:531-616, 2026). Starting from a locally Hom-finite approximable triangulated category, we use the Brown--Comenetz duals of compact objects to construct a triangulated subcategory $\E$, which plays the role of an injective-side analogue of the compact subcategory $\T^c$. We introduce the intrinsic subcategory $\T_c^+$, dual to Neeman's subcategory $\T_c^-$, and characterize its objects by strong $\E$-coapproximating systems and homotopy inverse limits. Under the compact silting hypothesis, we prove Brown representability theorems identifying $(\T_c^+)^{\op}$ with locally finite $\E$-homological functors and $(\T_c^b)^{\op}$ with finite $\E$-homological functors. We also establish localization results for recollements on the Brown--Comenetz side and derive applications to derived categories of finite-dimensional algebras.
\end{abstract}
	\maketitle
	
	\tableofcontents

    \section{Introduction}
Recollements and representability theorems lie at the heart of modern homological algebra and the study of derived categories \cite{AKLY17,BBD,Chen2,N18a}. Neeman's work on compactly generated triangulated categories \cite{N92,N18a} clarified how compact objects control representability, approximability and localization phenomena. In particular, for an approximable triangulated category $\T$, the intrinsic subcategory $\T_c^-$ is governed by compact objects and locally finite cohomological functors on $\T^c$, the subcategory of compact objects. More precisely, in his recent work \cite{N18a}, Neeman shows that
approximability allows one to construct intrinsic subcategories
\[
   \mathcal{T}^{-}_{c}
   =
   \bigcap_{n>0}\bigl(\mathcal{T}^{c} * \mathcal{T}^{\leq -n}\bigr),
   \qquad
   \mathcal{T}^{b}_{c}
   =
   \mathcal{T}^{-}_{c}\cap \mathcal{T}^{b}.
\]
Neeman proves that, for a commutative noetherian ring $R$ and a locally Hom-finite approximable $R$-linear triangulated category,
the restricted Yoneda functor
\[
   \mathcal{T}^{-}_{c}
      \longrightarrow
   {\Hom}_{R}\bigl((\mathcal{T}^{c})^{\mathrm{op}}, R\text{-}\mathrm{Mod}\bigr),
   \qquad
   X\longmapsto{\Hom}_{\mathcal{T}}(-,X)|_{\mathcal{T}^{c}},
\]
is full and has as its essential image the locally finite cohomological
functors on $\mathcal{T}^{c}$. Moreover, its restriction to
$\mathcal{T}^{b}_{c}$ is fully faithful, and its essential image consists
precisely of the finite cohomological functors on $\mathcal{T}^{c}$.  Thus
Neeman's theorem gives an intrinsic Brown representability theorem for the
compact-side subcategories $\mathcal{T}^{-}_{c}$ and $\mathcal{T}^{b}_{c}$.

The starting point of this paper is the following question. In what follows, $k$ denotes a field. Let $A$ be a finite-dimensional algebra over $k$, and let $\mathcal{T}=\D(A)$ be the unbounded derived category of $A$. This is one of the basic examples of an approximable triangulated category.
In this case the subcategory of compact objects is
$\K^b(A\proj)$, the homotopy category of bounded complexes of finitely generated projective
$A$-modules.  It has a natural dual counterpart, namely
   $\K^b(A\inj)$,
the homotopy category of bounded complexes of finitely generated injective
$A$-modules.  Similarly, in this derived setting the subcategory
$\mathcal{T}^{-}_{c}$ can be identified with $\K^{-}(A\proj)$, the homotopy category of bounded-above complexes of finitely generated projective modules, while its natural
dual counterpart is $\K^{+}(A\inj)$, the homotopy category of bounded-below complexes of finitely generated injective modules. This example suggests the following general question.  

\begin{ques}
For an arbitrary approximable triangulated category
$\mathcal{T}$, what should be the correct dual counterparts of
$\mathcal{T}^c$ and $\mathcal{T}^{-}_{c}$?  Moreover, do these dual
subcategories admit Brown representability theorems analogous to those known
for $\mathcal{T}^c$ and $\mathcal{T}^{-}_{c}$?
\end{ques}

To look for a categorical dual of compactness inside a triangulated category, the notions of cocompact objects and 0-cocompact objects \cite{OPS19, OPS}  have been
introduced as formal or weak duals of compact objects. However, these notions do not capture the dual behaviour needed in the
basic examples coming from finite-dimensional algebras. Indeed, even in
the simplest case of \(k\), the derived category \(\D(k)\)
has no non-zero cocompact objects, while every object of \(\D(k)\) is 0-cocompact \cite[Example 2.2]{OPS}.  Thus neither cocompactness nor \(0\)-cocompactness
distinguishes the expected dual class of compact objects in this setting.

The point of departure of this paper is different.  Instead of dualizing the defining compactness condition itself, we dualize compact objects through Brown--Comenetz duality. More precisely, we replace the subcategory $\mathcal{T}^{c}$ by the subcategory generated by the Brown--Comenetz duals of compact objects.  For a compact object $C\in \mathcal{T}^{c}$, its Brown--Comenetz dual is an object $\mathbf{S}(C)$ representing the functor $\mathbf{D}{\Hom}_{\mathcal{T}}(C,-)$, namely $$\mathbf{D}{\Hom}_{\mathcal{T}}(C,-)\simeq \Hom_{\mathcal{T}}(-,\mathbf{S}(C)),$$ 
where
$\mathbf{D}=\operatorname{Hom}_{k}(-,k)$.
We   define $\mathcal{E}$ to be the smallest thick triangulated subcategory of $\mathcal{T}$ which contains all Brown--Comenetz duals of compact objects,
and regard $\mathcal{E}$ as the {\em Brown--Comenetz-side dual} of
$\mathcal{T}^{c}$. In the derived category of a finite-dimensional algebra, this construction recovers the expected passage from the homotopy category of bounded complexes of finitely generated projective modules to the homotopy category of bounded complexes of finitely generated injective modules. Thus the category $\mathcal{E}$ provides the
appropriate starting point for constructing the intrinsic subcategory $\T^+_c$, the dual counterpart of
$\mathcal{T}^{-}_{c}$, and for formulating the corresponding Brown representability theorems.
	
\subsection{Brown representability theorems}
The first main result of this paper is a Brown representability theorem for the
dual subcategories constructed above.  Formally, this theorem is the
Brown--Comenetz-side counterpart of Neeman's Brown representability theorem for
\(\mathcal T_c^{-}\) and \(\mathcal T_c^{b}\) \cite[Theorems~1.4(i)]{N18a}: the compact subcategory
\(\mathcal T^c\) is replaced by its Brown--Comenetz dual \(\mathcal E\), the
subcategory \(\mathcal T_c^{-}\) is replaced by \(\mathcal T_c^{+}\), and locally
finite cohomological functors on \(\mathcal T^c\) are replaced by locally finite
homological functors on \(\mathcal E\).
	
\begin{thm}[Theorem~\ref{thm:rep for T+c} and Theorem~\ref{thm: rep for Tbc}]
		Let $\T$ be a locally Hom-finite $k$-linear triangulated category with a compact silting object. 
		\begin{enumerate}
			\item The restricted Yoneda functor
			\[
			y: {(\mathcal{T}_c^+)}^{\op} \longrightarrow \mathsf{Hom}_k(\mathcal{E}, k\text{-}\mathsf{Mod}),\quad T\longmapsto \Hom_{\T}(T,-)|_{\E},
			\]
			is full, and the essential image consists of all locally finite $\mathcal{E}$-homological functors.
			\item The restricted Yoneda functor
			\[
			y: {(\mathcal{T}_c^b)}^{\op} \longrightarrow \mathsf{Hom}_k(\E, k\text{-}\mathsf{Mod}),\quad T\longmapsto \Hom_{\T}(T,-)|_{\E},
			\]
			is fully faithful, and the essential image consists of all finite $\E$-homological functors.
	\end{enumerate}
\end{thm}
The proof does not follow from a purely formal dualization of Neeman's
argument.  Several steps require new ingredients specific to the Brown--Comenetz side:
compact approximations are replaced by coapproximations by objects of
\(\mathcal E\), homotopy colimits are replaced by homotopy limits, and the
necessary control of degrees depends on the compact silting hypothesis (Lemma~\ref{lem:key2}). Indeed, the category $\T$ under the assumption of having a compact silting object is approximable (see \cite[Remark 5.3]{N18a}).
	
\subsection{Localization theorems}
The first localization theorem is a dual Neeman theorem: for a recollement of compactly generated, locally Hom-finite triangulated categories, the third row restricts up to the Brown--Comenetz duals of the subcategories of compact objects. 
	\begin{thm}[Proposition~\ref{prop:dual neeman}]
	Let the following diagram be a recollement of compactly generated locally Hom-finite $k$-linear triangulated categories 
		$$\xymatrix{\R\ar^-{i_*=i_!}[r]
			&\T\ar^-{j^!=j^*}[r]\ar^-{i^!}@/^1.2pc/[l]\ar_-{i^*}@/_1.6pc/[l]
			&\S.\ar^-{j_*}@/^1.2pc/[l]\ar_-{j_!}@/_1.6pc/[l]}$$
		Then the third row restricts to a short exact sequence up to direct summands
		$$\E_s\stackrel{j_*}\longrightarrow \E_t\stackrel{i^!}\longrightarrow \E_r$$
		where $\E_s, \E_t, \E_r$ are the Brown--Comenetz duals of $\S^c, \T^c, \R^c$, respectively.    
	\end{thm}
	In particular, for recollements of derived categories of finite-dimensional algebras, the third row restricts to a short exact sequence of homotopy categories of bounded complexes of finitely generated injective modules.
	
As an important application of the Brown representability theorem above, we obtain a localization theorem for the intrinsic subcategories \(\mathcal T_c^{+}\) and \(\mathcal E\).  This result should be viewed as the Brown--Comenetz dual version of the localization theorem of Sun--Zhang--Zhang \cite{SZZ}.  Whereas their theorem concerns the compact-side subcategories and the corresponding Verdier quotient categories, our result establishes the analogous localization statements on the dual side.

	\begin{thm}[Theorem~\ref{thm:loc for T+c} and Theorem~\ref{thm:loc for E}]
		Let the following diagram be a recollement of locally Hom-finite $k$-linear triangulated categories, each of which admits a compact silting object.
		$$\xymatrix{\R\ar^-{i_*=i_!}[r]
			&\T\ar^-{j^!=j^*}[r]\ar^-{i^!}@/^1.2pc/[l]\ar_-{i^*}@/_1.6pc/[l]
			&\S.\ar^-{j_*}@/^1.2pc/[l]\ar_-{j_!}@/_1.6pc/[l]}$$
		Then 
		\begin{enumerate}
			\item Suppose that $\T$ has a compact generator $G$ such that there is an integer $N$ with ${\Hom}_{\T}(G, G[n])=0, n<N$. If the recollement can extend one step upwards, then the first row induces a short exact sequence $${\S}^+_c/{\S}^c\stackrel{\overline{j_!}}{\longrightarrow}\T^+_c/\T^c\stackrel{\overline{i^*}}{\longrightarrow}{\R}^+_c/{\R}^c.$$
			\item The second row induces a short exact sequence    $${\R}^+_c/{\R}^b_c\stackrel{\overline{i_*}}{\longrightarrow}\T^+_c/\T^b_c\stackrel{\overline{j^*}}{\longrightarrow}{\S}^+_c/{\S}^b_c.$$
			\item The third row induces a short exact sequence
			$${\S}^+_c/\E_s\stackrel{\overline{j_*}}{\longrightarrow}\T^+_c/\E_t\stackrel{\overline{i^!}}{\longrightarrow}{\R}^+_c/\E_r.$$
		\end{enumerate}
	\end{thm}
	
	The paper is organized as follows. Section~\ref{sec:pre} recalls notation, $t$-structures, short exact sequences of triangulated categories, and the approximable framework. Section~\ref{sec:BC dual} develops Brown--Comenetz duality and proves the dual Neeman theorem. Section~\ref{sec:int subcat} defines $\T^+_c$, proves its homotopy-limit characterization, and establishes basic finiteness properties. Section~\ref{sec:rep thm} contains the representability theorems for $\T^+_c$ and $\E_t$. Section~\ref{sec:loc thm} gives localization results for $\T^+_c$ and $\E_t$, and their related Verdier quotient categories.
	
	\section{Preliminaries}\label{sec:pre}
	In this section, we collect the notation, definitions, and basic facts. Throughout this paper, $k$ denotes a field.
	
	\subsection{Notation}
	Let $\T$ be a triangulated category. For two subcategories $\U, \V$ of $\T$, we set
	$$\U*\V:=\{T\in\T\mid T~\text{admits a triangle}~U\to T\to V, U\in\U, V\in\V\}.$$ We also define $$\mathcal{U}^{\perp}\coloneqq\{M\in \mathcal{T}\mid \Hom_{\mathcal{T}}(U,M)=0,\ \forall U\in \mathcal{U}\}$$ as the full subcategory of $\mathcal{T}$ right orthogonal to $\mathcal{U}$. The left orthogonal to $\mathcal{U}$ is defined dually.
	
    Assume $\T$ has small coproducts and let $G$ be an object of $\T$. For $a< b\in\mathbb{Z}\cup\{-\infty, +\infty\}$, 
	we denote by $\overline{\langle G\rangle}^{[a,b]}$ the smallest subcategory of $\T$ which contains $G[-i]$ where $a\leq i\leq b$ and is closed under direct summands, coproducts and extensions. For a positive integer $n$, the subcategories $\langle G\rangle_n^{[a,b]}$ and $\overline{\langle G\rangle}_n^{[a,b]}$ are defined inductively by
	\begin{align*}
        \langle G\rangle_1^{[a,b]} &= \text{direct summands of finite direct sums of } \{\,G[-i] \mid a \le i \le b \,\},\\
        \langle G\rangle_n^{[a,b]} &= \text{direct summands of objects in } 
		\langle G\rangle_1^{[a,b]} * \langle G\rangle_{n-1}^{[a,b]},\\
		\overline{\langle G\rangle}_1^{[a,b]} &= \text{direct summands of coproducts of } \{\,G[-i] \mid a \le i \le b \,\},\\
		\overline{\langle G\rangle}_n^{[a,b]} &= \text{direct summands of objects in } 
		\overline{\langle G\rangle}_1^{[a,b]} * \overline{\langle G\rangle}_{n-1}^{[a,b]}.
	\end{align*}
	
	We denote $\langle G\rangle:=\bigcup_{i\geq 1}\langle G\rangle_{i}^{[-i,i]}$. Clearly $\langle G\rangle$ is the smallest thick subcategory of $\T$ containing $G$. For simplicity, we denote $\langle G\rangle_{n}^{(-\infty,+\infty)}$ by $\langle G\rangle_{n}$. Dually, for a triangulated category $\T$ having products, and an object $E \in \T$, we define $\overline{[E]}_n^{[a,b]}$ by replacing coproducts with products in the definition of $\overline{\langle G\rangle}_n^{[a,b]}$. It is well-known that a compactly generated triangulated category has coproducts and products.
	
	Let $R$ be a ring. We denote by $R\Mod$ (resp., $R\smod$, $R\proj$) the category of right $R$-modules (resp., finitely presented right $R$-modules, finitely generated projective right $R$-modules).  
	We write $\D(R)$ (resp., $\D^b(R\smod)$, $\D^-(R\smod)$, $\K^b(R\proj)$, $\K^{-,b}(R\proj)$) for the derived category (resp., bounded derived category, upper-bounded derived category, homotopy category of bounded complexes, homotopy category of upper-bounded complexes with bounded cohomology) of the corresponding module categories.
	
	\subsection{\texorpdfstring{$t$}{t}-structures and homological facts for triangulated categories}
	We briefly recall some basic notions concerning $t$-structures.  
	A pair $(\T^{\le 0}, \T^{\ge 0})$ of full subcategories of a triangulated category $\T$ is called a \emph{$t$-structure} \cite[Definition~1.3.1]{BBD} if:
	\begin{enumerate}
		\item $\T^{\leq 0}[1]\subseteq \T^{\leq 0}$ and $\T^{\geq 0}[-1]\subseteq \T^{\geq 0}$;
		\item ${\Hom}_{\T}(\T^{\leq 0}, \T^{\geq 0}[-1])=0$;
		\item For any $T\in\T$, there exists a triangle
		$$U\longrightarrow T\longrightarrow V\longrightarrow U[1],$$
		where $U\in\T^{\leq 0}$ and $V\in\T^{\geq 0}[-1]$.
	\end{enumerate}

    For each $n \in \mathbb{Z}$, we denote $$\T^{\leq n}\coloneqq \T^{\leq 0}[-n]\;\;\text{and}\;\;\T^{\geq n}\coloneqq \T^{\geq 0}[-n].$$ The \emph{heart} of the $t$-structure, $\mathcal{H} := \T^{\le 0} \cap \T^{\ge 0}$, is an abelian category and yields a cohomological functor $\mathbf{H}^0 : \T \to \mathcal{H}$.  
	For $i \in \mathbb{Z}$, we set $\mathbf{H}^i := \mathbf{H}^0 \circ [i]$.
	
	Two $t$-structures $(\T^{\le 0}_1, \T^{\ge 0}_1)$ and $(\T^{\le 0}_2, \T^{\ge 0}_2)$ are said to be \emph{equivalent} if there exists $n \in \mathbb{N}$ such that
	\[
	\T^{\le -n}_1 \subseteq \T^{\le 0}_2 \subseteq \T^{\le n}_1.
	\]
	
Assume $\T$ and $\T'$ are triangulated categories with $t$-structures $t=(\T^{\leq 0}, \T^{\geq 0})$ and $t'=(\T'^{\leq 0},\T'^{\geq 0})$ respectively. Let $F:\T\to \T'$ be a triangulated functor, we say $F$ is \emph{left $t$-bounded} with respect to $(t,t')$ if there exists $n\geq 0$ such that $F(\T^{\leq 0})\subset \T'^{\leq n}$. One can define \emph{right $t$-bounded} functors similarly.

	When $\T$ has coproducts and a compact generator $G$, the results of \cite[Theorem~A.1]{ALS03} and \cite[Theorem~2.3]{N18d} show that $G$ determines a canonical $t$-structure
	\[
	(\T^{\le 0}_G, \T^{\ge 0}_G)
	:= \bigl( \overline{\langle G\rangle}^{(-\infty,0]},\, (\overline{\langle G\rangle}^{(-\infty,-1]})^\perp \bigr),
	\]
	whose equivalence class is called the \emph{preferred equivalence class}.
    
    In a triangulated category with a compact generator $G$, we consider the subcategories
	\[
	\T^- \coloneqq \bigcup_{n\in\mathbb{Z}} \T_{G}^{\le n}, \qquad
	\T^+ \coloneqq \bigcup_{n\in\mathbb{Z}} \T_{G}^{\ge n}, \qquad
	\T^b \coloneqq \T^- \cap \T^+,
	\]
	and those defined via compact approximations \[\T^-_c \coloneqq \bigcap_{i=1}^{\infty} (\T^{c} * \T_{G}^{\leq -i}), \qquad \T^b_c \coloneqq \T_{c}^{-} \cap \T^{b}.\]
	
	\begin{defn}\label{ML}
		A sequence \[\cdots \longrightarrow E_3 \longrightarrow E_2 \longrightarrow E_1 \] in $\mathbf{Ab}$ satisfies the \emph{Mittag-Leffler condition} if for each $i$ there exists $j\geq i$ such that the image of $E_{s}\to E_{i}$ equals the image of $E_{j}\to E_{i}$ for all $s\geq j$.
	\end{defn}
    
	Clearly, if the sequence $\cdots\to E_{3}\to E_{2}\to E_{1}$ consists of finite-dimensional vector spaces over a field $k$, then it automatically satisfies the Mittag-Leffler condition.

	\begin{lem}\label{ML lem1}
		If the sequence 
		$$\cdots\longrightarrow E_{3}\longrightarrow E_{2}\longrightarrow E_{1}$$ in $\mathbf{Ab}$ satisfies the Mittag-Leffler condition, then there exists an exact sequence in $\mathbf{Ab}${\rm :}\[0\longrightarrow\lims(E_{*})\longrightarrow \prod_{k=1}^{\infty}E_{k}\xrightarrow{1-p}\prod_{k=1}^{\infty}E_{k}\longrightarrow 0.\]
	\end{lem}
	
	\begin{lem}\label{ML lem2}
		Let $\T$ be a $k$-linear triangulated category and 
		$$\cdots\longrightarrow E_{3}\longrightarrow E_{2}\longrightarrow E_{1}$$ a sequence in $\T$. For an object $M \in \T$ such that ${\Hom}_\T(M,E_{i})$ is finite-dimensional for all $i>0$, there is a canonical isomorphism $$\lims\Hom_\T(M,E_{*})\simeq{\Hom}_\T(M,\Holim(E_{*})).$$
	\end{lem}
	
	\begin{lem}\label{tech1}
		Let $(\T^{\leq 0}, \T^{\geq 0})$ be a $t$-structure on $\T$. For $n \in \mathbb{Z}$, if $F \in \T^{+}$ satisfies $\mathbf{H}^{i}(F) = 0$ for all $i < n$, then $F \in \T^{\geq n}$.
	\end{lem}
	\begin{proof}
		Assume $F \in \T^{\geq k}$ for some $k < n$. Then there exists a triangle $$\mathbf{H}^{k}(F)[-k]\longrightarrow F\longrightarrow F^{\geq k+1}\longrightarrow \mathbf{H}^{k}(F)[-k+1]$$
		where $\mathbf{H}^{k}(F)[-k]=0=\mathbf{H}^{k}(F)[-k+1]$. Thus we have
		$F \simeq F^{\geq k+1}$, and consequently $F \in \T^{\geq n}$.
	\end{proof}

      The following lemma is well-known and useful in the proof of our main theorems. For the convenience of the reader, we give a proof here. 
	
	\begin{lem}\label{tech3}
		Let $\T$ be a compactly generated triangulated category equipped with a $t$-structure $(\T^{\leq 0}, \T^{\geq 0})$ such that $\T^{\leq 0}$ is closed under products.  Consider a sequence \[ (E_{*},f_{*}) :\quad \cdots \xrightarrow{f_3} E_3 \xrightarrow{f_2} E_2 \xrightarrow{f_1} E_1 \] in $\T$.  If for every integer $i$ there exists $N_{i} > 0$ such that $\mathbf{H}^{i}(f_j)$ is an isomorphism for all $j > N_{i}$, then  for all $i\in\mathbb{Z}$, there is a canonical isomorphism
		$$\mathbf{H}^{i}(\Holim(E_{*}))\simeq\lims(\mathbf{H}^{i}(E_{*})).$$
	\end{lem}
	\begin{proof}
		By definition of the homotopy limit, there is a triangle
		\[
		\Holim(E_{*}) \longrightarrow \prod_{k=1}^{\infty} E_{k}
		\stackrel{1-p}{\longrightarrow} \prod_{k=1}^{\infty} E_{k}
		\longrightarrow \Holim(E_{*})[1].
		\]		
		Applying $\mathbf{H}^{i}(-)$ gives the long exact sequence
		{\xiaoliuhao
			\[
			\mathbf{H}^{i-1}\!\left(\prod_{k=1}^{\infty} E_{k}\right)
			\xrightarrow{\mathbf{H}^{i-1}(1-p)}
			\mathbf{H}^{i-1}\!\left(\prod_{k=1}^{\infty} E_{k}\right)
			\longrightarrow
			\mathbf{H}^{i}(\Holim(E_{*}))
			\longrightarrow
			\mathbf{H}^{i}\!\left(\prod_{k=1}^{\infty} E_{k}\right)
			\xrightarrow{\mathbf{H}^{i}(1-p)}
			\mathbf{H}^{i}\!\left(\prod_{k=1}^{\infty} E_{k}\right).
			\]
		}
		
		Since $\ker(\mathbf{H}^{i}(1-p)) \simeq \lims(\mathbf{H}^{i}(E_{*}))$, it remains to show that $\mathbf{H}^{i-1}(1-p)$ is surjective. Consider the inverse system
		\[
		\cdots \longrightarrow \mathbf{H}^{i-1}(E_{2}) \longrightarrow \mathbf{H}^{i-1}(E_{1}).
		\]Then we obtain a short exact sequence
		\[
		0 \to \lims(\mathbf{H}^{i-1}(E_{*}))
		\to \prod_{k=1}^{\infty} \mathbf{H}^{i-1}(E_{k})
		\xrightarrow{\mathbf{H}^{i-1}(1-p)}
		\prod_{k=1}^{\infty} \mathbf{H}^{i-1}(E_{k}).
		\]		
		For any \(H\in\mathcal{H}\), since \(\mathbf{H}^{i-1}(f_{j})\) is an isomorphism for all \(j>N_{i-1}\), the inverse system \(\Hom_{\T}(H,\mathbf{H}^{i-1}(E_{*}))\) satisfies the Mittag–Leffler condition in \(\mathbf{Ab}\). Hence, by Lemma~\ref{ML lem1}, applying the functor \(\Hom_{\T}(H,-)\) yields the exact sequence:
		\[
			0 \to {\T}(H, \lims(\mathbf{H}^{i-1}(E_{*})))
			\to {\T}\!\left(H, \prod_{k=1}^{\infty} \mathbf{H}^{i-1}(E_{k})\right)
			\to {\T}\!\left(H, \prod_{k=1}^{\infty} \mathbf{H}^{i-1}(E_{k})\right)
			\to 0.
			\]
			Since $\T^{\leq 0}$ is closed under products, so is the heart $\mathcal{H}$, and then
             the product $\prod_{k=1}^{\infty} \mathbf{H}^{i-1}(E_{k})$ lies in $\mathcal{H}$. One can take $H=\prod_{k=1}^{\infty} \mathbf{H}^{i-1}(E_{k})$, hence $\mathbf{H}^{i-1}(1-p)$ is a split epimorphism.
	\end{proof}

	\subsection{Short exact sequences and recollements of triangulated categories}
	In analogy with the notion of short exact sequences in abelian categories, a sequence of triangulated categories
	$$\R\stackrel{F}{\longrightarrow}\T\stackrel{G}{\longrightarrow}\S$$
	is a {\em short exact sequence} if it satisfies
	\begin{enumerate}
		\item $F$ is fully faithful;
		\item $GF=0$; and
		\item the induced functor $\overline{G}: \T/\R\to \S$ is an equivalence.
	\end{enumerate}
	
	The triangulated analogue of the classical $3 \times 3$ lemma \cite[Lemma~2.5]{SZZ} will be used frequently.
	
	\begin{defn}(\cite[1.4.3]{BBD})
		Let $\T$, $\R$ and $\S$ be triangulated categories.
		We say that $\T$ is a {\em recollement} of $\R$
		and $\S$ if there is a diagram of six triangle functors
		$$\xymatrix{\R\ar^-{i_*=i_!}[r]&\T\ar^-{j^!=j^*}[r]
			\ar^-{i^!}@/^1.2pc/[l]\ar_-{i^*}@/_1.6pc/[l]
			&\S\ar^-{j_*}@/^1.2pc/[l]\ar_-{j_!}@/_1.6pc/[l]}$$ such
		that
		\begin{enumerate}
			\item $(i^*,i_*),(i_!,i^!),(j_!,j^!)$ and $(j^*,j_*)$ are adjoint
			pairs;
			\item $i_*,j_*$ and $j_!$ are fully faithful functors;
			\item $i^!j_*=0$; and
			\item for each object $T\in\T$, there are two triangles in
			$\T$
			$$
			i_!i^!(T)\to T\to j_*j^*(T)\to i_!i^!(T)[1],
			$$
			$$
			j_!j^!(T)\to T\to i_*i^*(T)\to j_!j^!(T)[1].
			$$
		\end{enumerate}
	\end{defn}
	
	A recollement is said to \emph{extend one step downwards} (resp., \emph{upwards}) if additional adjoint functors exist so that the diagram extends accordingly.
	
\subsection{Approximable triangulated categories}

We conclude by recalling the notion of approximability introduced in \cite{N18a,N21b}, which has become a powerful tool for studying Bondal–Van den Bergh’s Conjecture and Rouquier’s Strong Generation Conjecture \cite{N21a},  
in Rickard’s theorem on derived equivalences \cite{CHNS24}, and in recent progress on the existence of bounded $t$-structures \cite{BCPRZ24,N22a}, and also provides the foundational framework underlying the representability and localization results developed in this paper. Let $R$ be a commutative noetherian ring and let $\T$ be an $R$-linear triangulated category with a compact generator $G$. A functor $\mathbf{H} : (\T^c)^{\op} \to R\text{-}\mathsf{Mod}$ is called \emph{$\T^c$-cohomological} if it sends triangles to long exact sequences.  
It is said to be \emph{locally finite} (resp.\ \emph{finite}) if $\mathbf{H}(T)$ is finitely generated for all $T \in \T^c$ and vanishes in sufficiently negative (resp.\ sufficiently large positive and negative) degrees.  
The category $\T$ is \emph{locally Hom-finite} if $\Hom_\T(X,Y)$ is finitely generated over $R$ for all compact objects $X,Y \in \T^c$.

\begin{defn}(\cite[Definition 1.25]{N18a})\label{def:appro}
A triangulated category $\T$ with coproducts is called \emph{approximable} if it admits a compact generator $G$ and an integer $n>0$ such that:
\begin{itemize}
    \item[(i)] $\Hom_\T(G,G[i]) = 0$ for all $i \ge n$;
    \item[(ii)] every object $X \in \T^{\le 0}_G$ fits into a triangle

\[
    E \longrightarrow X \longrightarrow D \longrightarrow E[1]
    \]

    with $E \in \overline{\langle G\rangle}_{n}^{[-n,n]}$ and $D \in \T^{\le -1}_G$.
\end{itemize}
\end{defn}

This notion is independent of the choice of compact generator and of the representative of the preferred equivalence class of $t$-structures \cite[Facts~1.26]{N18a}.

	\begin{thm}\textnormal{(\cite[Theorem~1.4]{N18a})}\label{lem:object in bc}
	Let $\T$ be a locally Hom-finite approximable $R$-linear triangulated category. Consider:
    \begin{align*}
        Y: \T^-_c &\longrightarrow \Hom((\T^c){}^{\op}, R\Mod) \\
        X &\longmapsto \Hom(-,X)|_{\T^{c}}\\
        Y\circ i: {\T^b_c} &\longrightarrow \Hom((\T^c){}^{\op}, R\Mod) 
    \end{align*}
\begin{enumerate}
    \item $Y$ is full, and its essential image consists of all locally finite functors.
    \item $Y\circ i$ is fully faithful, and its essential image consists of all finite functors.
\end{enumerate}
	\end{thm}
	
	\section{Brown--Comenetz duality of compact objects}\label{sec:BC dual}
	
	In this section, we recall the Brown--Comenetz duality for compact objects in a triangulated category and establish a dual form of Neeman's classical theorem on recollements. The category $\E$ constructed below is the injective-side analogue of the compact subcategory $\T^c$. Throughout this section, let $\T$ be a compactly generated $k$-linear triangulated category, and let $\mathbf{D}$ denote the $k$-linear duality ${\Hom}_k(-,k)$.
	
	For a compact object $G \in \T$, the Brown representability theorem \cite{Br65,N96} implies that the cohomological functor $\mathbf{D}{\Hom}_{\T}(G,-)$ is representable. Hence, there exists an object $E \in \T$ and a canonical isomorphism
    \[
	{\Hom}_{\T}(-,E) \simeq \mathbf{D}{\Hom}_{\T}(G,-).
	\]
	The object $E$ is called the \emph{Brown--Comenetz dual} of $G$, a notion originating in \cite{BZ76}. Brown--Comenetz duality extends functorially: it induces a triangle functor $\mathbf{S}:\T^{c}\to\T$, known as the \emph{partial Serre functor} (see \cite[Theorem~3.3]{OPS}), such that
	\[
	{\Hom}_{\T}(-,\mathbf{S}(G)) \simeq \mathbf{D}{\Hom}_{\T}(G,-).
	\]
    Let $\operatorname{Im}(\mathcal{S})$ denote the essential image of $\T^{c}$ under $\mathbf{S}$, and let $\E$ be the smallest thick subcategory of $\T$ containing $\operatorname{Im}(\mathcal{S})$. We refer to $\E$ as the \emph{Brown--Comenetz dual} of the subcategory $\T^c$.

    A typical example arises from finite-dimensional algebras: if $A$ is a finite-dimensional $k$-algebra, then $\D(A)$ is a locally Hom-finite triangulated category with a compact generator $A$. In this case, $\T^c = \K^b(A\proj)$ and its Brown--Comenetz dual is $\K^b(A\inj)$, the homotopy category of bounded complexes of finitely generated injective modules.

    \begin{lem}\label{lem-ebasic}
    The following statements hold.
    \begin{enumerate}
        \item If $\T$ admits a compact generator $G$, then $\E=\langle\mathbf{S}(G)\rangle$.
        \item If $\T$ is locally Hom-finite, then $\mathbf{S}:\T^{c}\to\E$ is a triangle equivalence. In particular, $\E={\rm Im}(\mathbf{S})$.
    \end{enumerate}
    \end{lem}
    \begin{proof}
    (1) In this case, $\T^{c}=\langle G\rangle$. Since $\mathbf{S}$ is a triangle functor, we have ${\rm Im}(\mathbf{S})\subseteq\langle\mathbf{S}(G)\rangle$. Note that $\langle\mathbf{S}(G)\rangle$ is a thick subcategory of $\T$. Thus $\E\subseteq\langle\mathbf{S}(G)\rangle$. The inclusion $\langle\mathbf{S}(G)\rangle\subseteq\E$ is clear.

    (2) By \cite[Observation~3.4]{OPS}, the functor $\mathbf{S}$ in this case is fully faithful. Hence we obtain a triangle equivalence $\mathbf{S} : \T^c \xrightarrow{\sim} {\rm Im}(\mathbf{S})$. Thus it suffices to show $\mathrm{Im}(\mathbf{S})=\E$. By \cite[Proposition~2.1.1]{Bv03}, $\T^{c}$ is Karoubian, i.e., every idempotent morphism in $\T^{c}$ is split in $\T^{c}$. Consequently, $\mathrm{Im}(\mathbf{S})$ is also Karoubian and therefore closed under direct summands. Hence $\mathrm{Im}(\mathbf{S})$ is a thick subcategory of $\T$, and we obtain $\mathrm{Im}(\mathbf{S})=\E$.
    \end{proof}
	
	We now recall a fundamental result of Neeman concerning recollements of compactly generated triangulated categories.
	
	\begin{lem}\textnormal{(\cite[Theorem 2.1]{N92})}\label{lem:reco to c}
		Let
		\[
		\xymatrix{
			\R \ar[r]^{i_* = i_!} &
			\T \ar[r]^{j^! = j^*} \ar@/^1.2pc/[l]^{i^!} \ar@/_1.6pc/[l]_{i^*} &
			\S \ar@/^1.2pc/[l]^{j_*} \ar@/_1.6pc/[l]_{j_!}
		}
		\]
		be a recollement of compactly generated triangulated categories. Then the recollement restricts, up to direct summands, to a short exact sequence
		\[
		\S^c \xrightarrow{\, j_! \,} \T^c \xrightarrow{\, i^* \,} \R^c.
		\]
	\end{lem}
	
	The next proposition provides a dual counterpart of Neeman’s theorem, formulated on the level of Brown--Comenetz duals of compact objects.
	
	\begin{prop}\label{prop:dual neeman}
		Let
		\[
		\xymatrix{
			\R \ar[r]^{i_* = i_!} &
			\T \ar[r]^{j^! = j^*} \ar@/^1.2pc/[l]^{i^!} \ar@/_1.6pc/[l]_{i^*} &
			\S \ar@/^1.2pc/[l]^{j_*} \ar@/_1.6pc/[l]_{j_!}
		}
		\]
		be a recollement of compactly generated, locally Hom-finite $k$-linear triangulated categories. Then the third row restricts, up to direct summands, to a short exact sequence
		\[
		\E_s \xrightarrow{\, j_* \,} \E_t \xrightarrow{\, i^! \,} \E_r,
		\]
		where $\E_s, \E_t, \E_r$ denote the Brown--Comenetz duals of $\S^c, \T^c, \R^c$, respectively.
	\end{prop}
	
	\begin{proof}
		We claim that the following diagram commutes up to natural isomorphism:
		\[
		\xymatrix{
			\S^c \ar[r]^{j_!} \ar[d]_{\mathbf{S}_s} &
			\T^c \ar[r]^{i^*} \ar[d]_{\mathbf{S}_t} &
			\R^c \ar[d]_{\mathbf{S}_r} \\
			\E_s \ar[r]_{j_*} &
			\E_t \ar[r]_{i^!} &
			\E_r,
		}
		\]
		where $\mathbf{S}_s, \mathbf{S}_t, \mathbf{S}_r$ are the corresponding partial Serre functors. We verify commutativity for the left square; the right square is analogous.
		
		For $G \in \S^c$, we have isomorphisms
		\begin{align*}
			{\Hom}_{\T}(-,\mathbf{S}_t j_!(G))
			&\simeq \mathbf{D}{\Hom}_{\T}(j_!G,-) && \text{(partial Serre functor)} \\
			&\simeq \mathbf{D}{\Hom}_{\S}(G,j^*(-)) && \text{(adjunction)} \\
			&\simeq {\Hom}_{\S}(j^*(-),\mathbf{S}_s(G)) && \text{(partial Serre functor)} \\
			&\simeq {\Hom}_{\T}(-,j_*\mathbf{S}_s(G)) && \text{(adjunction)}.
		\end{align*}
		By Yoneda’s lemma, this yields a natural isomorphism
		\[
		\mathbf{S}_t j_!|_{\S^c} \;\cong\; j_* \mathbf{S}_s|_{\S^c}.
		\]
		Thus the left square commutes up to isomorphism. Since the top row is exact up to direct summands by Lemma~\ref{lem:reco to c}, so is the bottom row.
	\end{proof}
	
	\begin{exam}
		Let
		\[
		\xymatrix{
			\D(B) \ar[r]^{i_* = i_!} &
			\D(A) \ar[r]^{j^! = j^*} \ar@/^1.2pc/[l]^{i^!} \ar@/_1.6pc/[l]_{i^*} &
			\D(C) \ar@/^1.2pc/[l]^{j_*} \ar@/_1.6pc/[l]_{j_!}
		}
		\]
		be a recollement of derived categories of finite-dimensional $k$-algebras.  
		Neeman’s theorem (Lemma~\ref{lem:reco to c}) yields a short exact sequence up to direct summands
		\[
		\K^b(C\proj) \xrightarrow{\, j_! \,} \K^b(A\proj) \xrightarrow{\, i^* \,} \K^b(B\proj).
		\]
		Our dual result (Proposition~\ref{prop:dual neeman}) provides the corresponding sequence on injectives:
		\[
		\K^b(C\inj) \xrightarrow{\, j_* \,} \K^b(A\inj) \xrightarrow{\, i^! \,} \K^b(B\inj).
		\]
	\end{exam}
	
	\section{The intrinsic subcategory $\T^+_c$}\label{sec:int subcat}
	The aim of this section is to introduce and study a new intrinsic subcategory $\T^+_c$ of a $k$-linear triangulated category $\T$ with a compact generator $G$.  
	This construction should be regarded as the formal dual of the classical subcategory $\T^-_c$ appearing in the theory of approximable triangulated categories.  
	While $\T^-_c$ plays a central role in controlling objects from above via compact approximations, the subcategory $\T^+_c$ provides a dual mechanism that approximates objects from below using the Brown--Comenetz duality.  
	This dual viewpoint will be essential for the structural results established in later sections.
	
	In this section, let $\T$ be a $k$-linear triangulated category with a compact generator $G$, and set $(\T^{\leq 0},\T^{\geq 0})=(\T_{G}^{\leq 0},\T_{G}^{\geq 0})$. Let $E$ be the Brown--Comenetz dual of $G$. Define
\[
\T_c^+ := 
\bigl\{
T \in \T \,\bigm|\,
\forall\, m>0,\ \exists\ \text{a triangle } 
D \to T \to F\ \text{with } F \in \E,\ D \in \T^{\ge m}
\bigr\},
\]
where $\E$ denotes the Brown--Comenetz dual of $\T^c$. Note that $\E=\langle E\rangle$ in this situation (see Lemma \ref{lem-ebasic}(1)).
	
	\begin{exam}
		Let $A$ be a finite-dimensional $k$-algebra and set $\T=\mathcal{D}(A)$. Then in this case one has $\T^+_c = \K^+(A\inj).$
	\end{exam}
	
	\begin{lem}\label{homfinitelem}
	    If $\T$ is locally Hom-finite, then 
		\begin{enumerate}
			\item ${\Hom}_\T(M,N)$ is finite-dimensional whenever $M\in\T^{c}$ and $N\in\T_{c}^{+}$.
			\item ${\Hom}_\T(L,K)$ is finite-dimensional whenever $L\in\T_{c}^{+}$ and $K\in\E$.
		\end{enumerate}
	\end{lem}
\begin{proof}
	(1) Since $\T^{c} = \langle G\rangle$, we have $\T^{c} \subseteq \T^{-}$.  
	Let $M \in \T^{c}$ and $N \in \T_{c}^{+}$.  
	There exists $k>0$ such that ${\Hom}_{\T}(M, \T^{\ge k}) = 0$.  
	By the definition of $\T_{c}^{+}$, we may choose a triangle
	\[
	D \longrightarrow N \longrightarrow F \longrightarrow D[1]
	\]
with $D \in \T^{\ge k+1}$ and $F \in \mathcal{E}$.  
	Applying ${\Hom}_{\T}(M,-)$ yields an isomorphism
	\[
	{\Hom}_{\T}(M,N) \simeq {\Hom}_{\T}(M,F).
	\]
	Thus it suffices to show that ${\Hom}_{\T}(M,F)$ is finite-dimensional for all  
	$M \in \T^{c}$ and $F \in \mathcal{E}$.
	
	Define
	\[
	\mathcal{B} \coloneqq 
	\{\, Y \in \T \mid \forall M \in \T^{c},\ {\Hom}_{\T}(M,Y)\ \text{is finite-dimensional} \,\}.
	\]
	Since ${\Hom}_{\T}(M,E) \simeq \mathbf{D}{\Hom}_{\T}(G,M)$ and $\mathcal{T}$ is locally finite, 
	the right-hand side is finite-dimensional. Thus $E \in \mathcal{B}$ and $\mathcal{B}$ is nonempty. As $\mathcal{B}$ is clearly a thick subcategory of $\mathcal{T}$, we conclude that $\mathcal{E} \subseteq \mathcal{B}$.
	
	(2) Define
	\[
	\mathcal{C} \coloneqq 
	\{\, Z \in \T \mid \forall L \in \T_{c}^{+},\ {\Hom}_{\T}(L,Z)\ \text{is finite-dimensional} \,\}.
	\]
	
	Clearly $\mathcal{C}$ is a thick subcategory of $\T$.  
	Thus it suffices to show that $E \in \mathcal{C}$.  
	Fix $X \in \T_{c}^{+}$.  
	We have
	\[
	{\Hom}_{\T}(X,E) \simeq \mathbf{D}{\Hom}_{\T}(G,X).
	\]
	By (1), ${\Hom}_{\T}(G,X)$ is finite-dimensional, and therefore so is  
	$\mathbf{D}{\Hom}_{\T}(G,X)$.  
	Hence $E \in \mathcal{C}$, completing the proof.
\end{proof}
	
	\begin{prop}\label{tech4}
		Assume there exists an integer $n>0$ such that $\Hom_{\T}(G,G[i])=0$ for $i\geq n$. Then the following statements hold.
		\begin{enumerate}
			\item ${\Hom}_\T(\T^{\geq 1}, E)=0$ and $\E\subseteq \T^+$. Moreover, If $E'\in\E$, then there exists $B>0$ such that ${\Hom}_\T(\T^{\geq B},E')=0$;
			\item $\T^+_c$ is a thick subcategory of $\T$, and contains $\E$.
		\end{enumerate}
	\end{prop}
\begin{proof}
	(1) Since $G\in\T^{\leq 0}$, we have ${\Hom}_{\T}(G, \T^{\ge 1}) = 0$. Then we obtain
	\[
	{\Hom}_{\T}(\T^{\ge 1}, E)
	\simeq \mathbf{D}{\Hom}_{\T}(G, \T^{\ge 1})
	= 0.
	\]
	For every \( i \in \mathbb{Z} \),
\[
	{\Hom}_{\T}(G[i], E)
	\simeq \mathbf{D}{\Hom}_{\T}(G, G[i]),
	\]
	and hence \( {\Hom}_{\T}(G[i], E) = 0 \) for \( i \geq n \).  
	By \cite[Corollary~3.10]{BNP18}, we obtain \( E \in \T^{+} \) and hence $\E\subseteq\T^{+}$.
	
	Now let \( E' \in \E \). By Lemma \ref{lem-ebasic}(1), there exists some \( k > 0 \) such that \( E' \in \langle E \rangle_{k}^{[-k, k]} \). Since \(\Hom_{\T}(\T^{\geq k+1}, E[i]) = 0\) for \(-k \leq i \leq k\) and \((\T^{\geq 1+k})^{\perp}\) is closed under finite direct sums, direct summands and extensions, it follows that \(\langle E \rangle_{k}^{[-k, k]} \subseteq (\T^{\geq 1+k})^{\perp}\). 
Consequently, \(\Hom_{\T}(\T^{\geq 1+k}, E') = 0\).
	
	(2) The conclusion follows from part (1) together with the strategy of Neeman’s proof  
	\cite[Proposition~3.10]{N18a}.
\end{proof}

	\begin{defn}\label{strong coapproximating system}
		Let $\mathcal{B}\subseteq\T$ and  \( (E_{*},f_{*}) : \cdots \xrightarrow{f_3} E_3 \xrightarrow{f_2} E_2 \xrightarrow{f_1} E_1 \) be a sequence in \(\mathcal{B}\).
		\begin{enumerate}
			\item It is called a \emph{strong \(\mathcal{B}\)-coapproximating system} if for each \( m \in \mathbb{N}^+ \), the morphism \( \mathbf{H}^{i}(f_{m}) \) is an isomorphism for all integers \( i \leq m \).
			\item Let \( F \in \T \). If there exist morphisms \( g_j: F \to E_j \) for all \( j \in \mathbb{N}^+ \) such that \( g_j = f_jg_{j+1} \) for all \( j \in \mathbb{N}^+ \), and for each \( m \in \mathbb{N}^+ \), the morphism \( \mathbf{H}^{i}(g_{m}) \) is an isomorphism for \( i \leq m \), then \( (E_{*},f_{*}) \) is called a \emph{strong \(\mathcal{B}\)-coapproximating system} for \( F \).
		\end{enumerate}
	\end{defn}
	
	\begin{lem}\label{universal property of coapproximating system}
		Assume that there exists an integer $n>0$ such that $\Hom_{\T}(G,G[i])=0$ for $i\geq n$ and that $\T^{\leq 0}$ is closed under products, we have:
		\begin{enumerate}
			\item If $(E_{*},f_{*})$ is a strong \(\mathcal{E}\)-coapproximating system, then it is a strong \(\mathcal{E}\)-coapproximating system for $\Holim(E_{*})$. Moreover, $\Holim(E_{*})\in\T_{c}^{+}$.
			\item Let $F \in \T^{+}$ and let $(E_{*}, f_{*})$ be a strong $\mathcal{E}$-coapproximating system for $F$. Then the induced morphism $F\to\Holim(E_{*})$ is an isomorphism.
		\end{enumerate}
	\end{lem}
\begin{proof}
	(1) By the definition of the homotopy limit, there is a triangle
	\[
	\Holim(E_{*}) \stackrel{g}{\longrightarrow} \prod_{k=1}^{\infty} E_{k}
	\longrightarrow \prod_{k=1}^{\infty} E_{k}
	\longrightarrow \Holim(E_{*})[1].
	\]
	Let \( p_{j} \colon \prod_{k=1}^{\infty} E_{k} \to E_{j} \) be the canonical projection and set  
	\( g_{j} \coloneqq p_{j} g \).  
	Then \( g_{j} = f_{j} g_{j+1} \) for all \( j>0 \).  
	For each \( m \in \mathbb{N}^{+} \), Lemma~\ref{tech3} shows that  
	\( \mathbf{H}^{i}(g_{m}) \) is an isomorphism whenever \( i \le m \).
	
	Since \( E_{1} \in \E \subseteq \T^{+} \) by Proposition \ref{tech4}(1), there exists \( n>0 \) such that  
	\( E_{1} \in \T^{\ge -n} \).  
	By the definition of the system \( (E_{*}, f_{*}) \), for any \( m \ge 1 \) and any \( i \le 1 \) we have  
	\( \mathbf{H}^{i}(E_{m}) \simeq \mathbf{H}^{i}(E_{1}) \).  
	Hence \( \mathbf{H}^{i}(E_{m}) = 0 \) for all \( i < -n \).  
	Since \( E_{m} \in \E \subseteq \T^{+} \), Lemma~\ref{tech1} implies  
	\( E_{m} \in \T^{\ge -n} \), and therefore  
	\( \Holim(E_{*}) \in \T^{+} \).
	
	Consider the triangle
	\[
	D_{m} \longrightarrow \Holim(E_{*}) \xrightarrow{g_{m}} E_{m} \longrightarrow D_{m}[1].
	\]
	Clearly \( D_{m} \in \T^{+} \).  
	Moreover, since \( \mathbf{H}^{i}(g_{m}) \) is an isomorphism for all \( i \le m \), we obtain  
	\( \mathbf{H}^{i}(D_{m}) = 0 \) for all \( i \le m \).  
	Applying Lemma~\ref{tech1} again yields  
	\( D_{m} \in \T^{\ge m+1} \).  
	Since \( m \) is arbitrary, the desired result follows.
	
	(2) There exists a morphism \( f \colon F \to \Holim(E_{*}) \) such that for every \( m \ge 1 \), the diagram
		\[
	\begin{tikzcd}
		F \arrow[d, "f"'] \arrow[rd] & \\
		\Holim(E_{*}) \arrow[r] & E_{m}
	\end{tikzcd}
	\]
	commutes.  
	For each \( i \in \mathbb{Z} \),  set \( N_{i} \coloneqq \max\{i+1,1\} \). Then the morphisms  
	\( \mathbf{H}^{i}(F) \to \mathbf{H}^{i}(E_{N_{i}}) \) and  
	\( \mathbf{H}^{i}(\Holim(E_{*})) \to \mathbf{H}^{i}(E_{N_{i}}) \)  
	are isomorphisms. Hence, $\mathbf{H}^{i}(f)$ is an isomorphism for all $i\in\mathbb{Z}$. By Lemma~\ref{tech1}, we deduce that $f$ is an isomorphism.
\end{proof}

	\begin{prop}\label{Tc+coapproximable}Assume that there exists an integer $n>0$ such that $\Hom_{\T}(G,G[i])=0$ for $i\geq n$ and that $\T^{\leq 0}$ is closed under products, we have:
		$$\T^+_c=\{\Holim(E_*)\mid E_*~\text{is a strong}~\E\text{-coapproximating system}\}.$$
	\end{prop}
\begin{proof}
	By Lemma~\ref{universal property of coapproximating system}, it suffices to show that every  
	\( F \in \T_{c}^{+} \) admits a strong \( \E \)-coapproximating system.  
	We construct such a system inductively.
	
	By definition, there exists a triangle
	\[
	D_{1} \longrightarrow F \stackrel{g_{1}}{\longrightarrow} E_{1} \longrightarrow D_{1}[1]
	\]
 with \( D_{1} \in \T^{\ge 3} \) and \( E_{1} \in \E \).  
	Applying the cohomological functor \( \mathbf{H}(-) \), we obtain an exact sequence
	\[
	\mathbf{H}^{i}(D_{1}) \longrightarrow \mathbf{H}^{i}(F) 
	\longrightarrow \mathbf{H}^{i}(E_{1}) 
	\longrightarrow \mathbf{H}^{i+1}(D_{1}).
	\]
	For all \( i \le 1 \), we have \( \mathbf{H}^{i}(D_{1}) = 0 = \mathbf{H}^{i+1}(D_{1}) \) since  
	\( D_{1} \in \T^{\ge 3} \).  
	Thus \( \mathbf{H}^{i}(F) \simeq \mathbf{H}^{i}(E_{1}) \) for all \( i \le 1 \).
	
	Now let \( n>1 \) and assume that morphisms  
	\( g_{i} \colon F \to E_{i} \) have been constructed for all \( 1 \le i \le n \).  
	By Proposition~\ref{tech4}(1), there exists \( B' > 0 \) such that  
	\( {\Hom}_{\T}(\T^{\ge B'}, E_{n}) = 0 \).  
	Set \( N = \max\{B',\, n+3\} \).  
	By definition, we obtain a triangle
	\[
	D_{n+1} \longrightarrow F \xrightarrow{g_{n+1}} E_{n+1} \longrightarrow D_{n+1}[1]
	\]
		with \( D_{n+1} \in \T^{\ge N} \) and \( E_{n+1} \in \E \).  
	For all \( i \le n+1 \), we still have  
	\( \mathbf{H}^{i}(F) \simeq \mathbf{H}^{i}(E_{n+1}) \).  
	Moreover, the vanishing \( {\Hom}_{\T}(D_{n+1}, E_{n}) = 0 \) induces a morphism  
	\( f_{n} \colon E_{n+1} \to E_{n} \) satisfying  
	\( g_{n} = f_{n} g_{n+1} \).  
	This completes the inductive step.
\end{proof}
	
	\begin{lem}\label{colim-holim1}
		Assume that there exists an integer $n>0$ such that $\Hom_{\T}(G,G[i])=0$ for $i\geq n$ and that $\T^{\leq 0}$ is closed under products. Let $(E_{*},f_{*})$ be a strong $\E$-coapproximating system. Then there exists a natural transformation $$\varphi:\colim\Hom_\T(E_{*},-)\longrightarrow{\Hom}_\T(\Holim(E_{*}),-)$$
		whose restriction to $\E$ is an isomorphism.
	\end{lem}
\begin{proof}
	By the definition of the homotopy limit, there is a triangle
	\[
	\Holim(E_{*}) \stackrel{g}{\longrightarrow} \prod_{k=1}^{\infty} E_{k}
	\longrightarrow \prod_{k=1}^{\infty} E_{k}
	\longrightarrow \Holim(E_{*})[1].
	\]
	Let \( p_{j} \colon \prod_{k=1}^{\infty} E_{k} \to E_{j} \) be the canonical projection and set  
	\( g_{j} \coloneqq p_{j} g \colon \Holim(E_{*}) \to E_{j} \).  
	Clearly \( g_{j} = f_{j} g_{j+1} \) for all \( j>0 \).  
	Since
	\[
	\T(g_{j}, -) = \T(g_{j+1}, -) \circ \T(f_{j}, -)
	\]
	for every \( j>0 \), the universal property of the colimit yields a natural transformation
	\[
	\varphi \colon \colim\Hom_{\T}(E_{*}, -) \longrightarrow {\Hom}_{\T}(\Holim(E_{*}), -).
	\]

	Define
	\[
	\mathcal{B} := \{\, B \in \T \mid \varphi(B) \text{ is an isomorphism} \,\}.
	\]
    By the five lemma, given a triangle $M \to N \to L \to M[1]$, if $M, M[1], L, L[1]$ all lie in $\mathcal{B}$, then $N$ also lies in $\mathcal{B}$. Thus it remains to show that \(\langle E \rangle_{1} \subseteq \mathcal{B} \). Since \( \mathcal{B} \) is closed under finite direct sums and direct summands,  
	it suffices to show that \( E[i] \in \mathcal{B} \) for all \( i \in \mathbb{Z} \).
	
	From the proof of Lemma~\ref{universal property of coapproximating system}(1),  
	for every \( m>0 \) there exists a triangle
	\[
	D_{m} \longrightarrow \Holim(E_{*}) \xrightarrow{g_{m}} E_{m} \longrightarrow D_{m}[1]
	\]
with \( D_{m} \in \T^{\ge m+1} \).  
	Fix \( i \in \mathbb{Z} \). As $G[i]\in\mathcal{T}^{-}$, the morphism \[\Hom_{\T}(G[i],\Holim(E_{*}))\xrightarrow{\T(G[i],g_{m})}\Hom_{\T}(G[i],E_{m})\]is an isomorphism for $m\gg0$. Since $\Hom_{\T}(-,E[i])\simeq\mathbf{D}{\Hom}_{\T}(G[i],-)$ for all $i\in\mathbb{Z}$, the morphism\[\Hom_{\T}(E_{m},E[i])\xrightarrow{\T(g_{m},E[i])}\Hom_{\T}(\Holim(E_{*}),E[i])\]is an isomorphism for $m\gg 0$.
	Consequently, $E[i]\in\mathcal{B}$ for all $i\in\mathbb{Z}$.
\end{proof}

	\begin{cor}\label{key1}
		Assume that there exists an integer $n>0$ such that $\Hom_{\T}(G,G[i])=0$ for $i\geq n$ and that $\T^{\leq 0}$ is closed under products. If $F\in\T_{c}^{+}$, then there exists a natural transformation \[\varphi:\colim\Hom_\T(E_{*},-)\longrightarrow {\Hom}_\T(F,-)\] 
		such that the restriction $\varphi|_{\E}$ is an isomorphism, where $E_{*}$ is a strong $\E$-coapproximating system for $F$.
	\end{cor}
	
	\section{Representability theorems for $\T^+_c$ and $\T^b_c$}\label{sec:rep thm}
	
	This section is to develop representability theorems for the intrinsic subcategory $\T^+_c$, 
	which was introduced in the previous section as the formal dual counterparts of the classical subcategory 
	$\T^-_c$ appearing in Neeman’s theory of approximable triangulated categories.  
	While Neeman’s representability results for $\T^-_c$ \cite[Theorem 9.18]{N18a} rely on homotopy colimits and compact approximations, 
	the dual theory for $\T^+_c$ involves homotopy limits and Brown--Comenetz duality, and is therefore technically more delicate.  
	In particular, the dual arguments do not mirror the original ones in a straightforward manner. 
   
    In this section, let $\T$ be an approximable $k$-linear triangulated category with a compact generator $G$, and set $(\T^{\leq 0},\T^{\geq 0})=(\T_{G}^{\leq 0},\T_{G}^{\geq 0})$. Let $E$ be the Brown--Comenetz dual of $G$ and let $\E$ be the Brown--Comenetz dual of $\T^{c}$.
	
	\subsection{Finite homological functors}
	
	Let $\mathcal{B}\subseteq\T$ be a full subcategory closed under isomorphisms such that $\mathcal{B}[1]=\mathcal{B}$. A $\mathcal{B}$-homological functor is a $k$-linear functor $\mathbf{H}\colon\mathcal{B}\to k\text{-}\mathsf{Mod}$ which sends triangles to long exact sequences (see \cite[Definition~1.1]{N18a}). Let $\mathbf{H}$ be a $\mathcal{B}$-homological functor.  
	We say that $\mathbf{H}$ is \emph{locally finite} (resp.\ \emph{finite}) if for every $T \in\mathcal{B}$:
	\begin{enumerate}
		\item $\dim_k \mathbf{H}(T) < \infty$, and
		\item $\mathbf{H}(T[i]) = 0$ for all $i \gg 0$ (resp.\ for all $|i| \gg 0$).
	\end{enumerate}

    \begin{rem}
    In Neeman's compact-side theory, a locally finite cohomological or homological functor is required to vanish in one shift direction, namely on sufficiently negative shifts of the chosen generator (\cite[Definition~1.1]{N18a}). In the Brown--Comenetz-side theory developed here, the homological functors on $\E$ are required to vanish in the opposite direction. 
	\end{rem}
	
	\begin{rem}\label{reduce}
		Since $\E = \langle E \rangle$ by Lemma \ref{lem-ebasic}(1), it follows from \cite[Remark~1.2]{N18a} that an $\E$-homological functor $\mathbf{H}$ is locally finite (resp.\ finite) if and only if 	
		\[
		\dim_k \mathbf{H}(E[j]) < \infty\;\text{for each}\;j\in\mathbb{Z}
		\quad\text{and}\quad
		\mathbf{H}(E[i]) = 0 \text{ for } i \gg 0 \text{ (resp.\ } |i| \gg 0).
		\]
	\end{rem}
	
	\begin{lem}\label{ker of homo functor}
		Let 
		\[
		0 \longrightarrow \mathbf{H}_1 \longrightarrow \mathbf{H}_2 \longrightarrow \mathbf{H}_3 \longrightarrow 0
		\]
		be an exact sequence in $\mathsf{Hom}_k(\E, k\text{-}\mathsf{Mod})$.  
		If $\mathbf{H}_2$ and $\mathbf{H}_3$ are $\E$-homological functors, then so is $\mathbf{H}_1$.  
		Moreover, if $\mathbf{H}_2$ and $\mathbf{H}_3$ are locally finite (resp.\ finite), then $\mathbf{H}_1$ is also locally finite (resp.\ finite).
	\end{lem}
	
	\begin{lem}\label{reasonable}
		Assume that $\T$ is locally Hom-finite. 
		\begin{enumerate}
			\item If $X\in\T_{c}^{+}$, then ${\Hom}_\T(X,-)|_{\E}$ is a locally finite $\E$-homological functor. 
			\item If $X\in\T_{c}^{+}\cap\T^{b}$, then ${\Hom}_\T(X,-)|_{\E}$ is a finite $\E$-homological functor. 
		\end{enumerate}
	\end{lem}
\begin{proof}
	(1) By Proposition~\ref{tech4}(1), there exists \( m>0 \) such that  
	\({\Hom}_{\T}(\T^{\ge m}, E) = 0\). Fix $j\in\mathbb{Z}$ and let $m'=\max\{m,m-j\}$. 
	Since \( X \in \T_{c}^{+} \), there is a triangle
	\[
	D_{m'+1} \longrightarrow X \longrightarrow E' \longrightarrow D_{m'+1}[1]
	\]
with \( D_{m'+1} \in \T^{\ge m'+1} \) and \( E' \in \E \).  
	Applying \( {\Hom}_{\T}(-,E[j]) \) to this triangle yields an isomorphism
	\[
	{\Hom}_{\T}(X,E[j]) \simeq {\Hom}_{\T}(E',E[j]).
	\]
	By Lemma~\ref{homfinitelem}(2), the latter is finite-dimensional, and hence  
	\({\Hom}_{\T}(X,E[j])\) is finite-dimensional as well. Since $X\in\T_{c}^{+}\subseteq\T^{+}$, it also follows from Proposition~\ref{tech4}(1) that $\Hom_{\T}(X,E[i])=0$ for $i\gg 0$. The conclusion then follows from Remark~\ref{reduce}.
	
	(2) By part (1), it remains to show that \( \Hom_{\T}(X,E[i]) = 0 \) for all \( i \ll 0 \). Proposition~\ref{tech4}(1) gives \( E \in \T^{+} \). Then, since \( X \in \T^{b} \subseteq \T^{-} \), we immediately obtain \( \Hom_{\T}(X,E[i]) = 0 \) for all \( i \ll 0 \).
\end{proof}

	\subsection{Preliminary representability theorem}

    In this subsection, we further assume that $\T^{\leq 0}$ is closed under product. 
	
	\begin{lem}\label{extend weak triangle}
		Let $f:M\to N$ be a morphism in $\T_{c}^{+}$. Then there exists a sequence $$N[-1]\longrightarrow L\longrightarrow M\stackrel{f}\longrightarrow N$$ in $\T_{c}^{+}$ such that
		\begin{enumerate}
			\item The sequence $${\Hom}_{\T}(N,-)|_{\E}\xrightarrow{\T(f,-)|_{\E}}{\Hom}_{\T}(M,-)|_{\E}\longrightarrow {\Hom}_{\T}(L,-)|_{\E}\longrightarrow {\Hom}_{\T}(N[-1],-)|_{\E}$$
			is a weak triangle {\rm (see \cite[Definition~8.2]{N18a})} in ${\Hom}_{R}(\E,k\text{-}\mathsf{Mod})$.
			\item For all $i\in\mathbb{Z}$, there exists an exact sequence 
			$$\cdots\longrightarrow \mathbf{H}^{i-1}(N)\longrightarrow \mathbf{H}^{i}(L)\longrightarrow \mathbf{H}^{i}(M)\longrightarrow \mathbf{H}^{i}(N).$$
            \item If $M$ admits a strong $\langle E\rangle_{n_{1}}$-coapproximating system and $N$ admits a strong $\langle E\rangle_{n_{2}}$-coapproximating system, then $L$ admits a strong $\langle E\rangle_{n_{1}+n_{2}}$-coapproximating system.
		\end{enumerate}
	\end{lem}
    \begin{proof}
    (1) and (3) are a specialization of \cite[Lemma~8.5]{N18a}. It can also be seen as a dual argument to \cite[Remark~9.7]{N18a}. (2) is a property of weak triangles (see \cite[Lemma~8.4]{N18a}).
    \end{proof}
	
	\begin{lem}\label{tech5}
		Let $\mathbf{H}$ be a locally finite $\langle E\rangle_{n}$-homological functor. Then there exists a sequence \[F_{n}\xrightarrow{f_{n-1}} F_{n-1}\xrightarrow{f_{n-2}}\cdots\stackrel{f_{2}}\longrightarrow F_{2}\stackrel{f_{1}}\longrightarrow F_{1}\]in $\T_{c}^{+}$ such that for each $i$
		\begin{enumerate}
			\item $F_{i}$ admits a strong $\langle E\rangle_{i}$-coapproximating system.
			\item There exists an epimorphism $$\varphi_{i}:{\Hom}_\T(F_{i},-)|_{\langle E\rangle_{i}}\longrightarrow \mathbf{H}|_{\langle E\rangle_{i}}$$  making the following diagram commutes
			\[\begin{tikzcd}
				{{\Hom}_\T(F_{i},-)|_{\langle E\rangle_{i}}} \arrow[rd, "\varphi_{i}"'] \arrow[rr, "{\T(f_{i},-)|_{\langle E\rangle_{i}}}"] &                                    & {{\Hom}_\T(F_{i+1},-)|_{\langle E\rangle_{i}}} \arrow[ld, "\varphi_{i+1}|_{\langle E\rangle_{i}}"] \\
				& \mathbf{H}|_{\langle E\rangle_{i}} &                                                                                            
			\end{tikzcd}.\]
			\item $\varphi_i$ in $(2)$ satisfies $\ker(\varphi_{i}|_{\langle E\rangle_{1}})=\ker(\T(f_{i},-)|_{\langle E\rangle_{1}})$.
		\end{enumerate}
	\end{lem}
	\begin{proof}
		We argue by induction on $n$. Let us begin with the case \(n=1\). By definition, \(\mathbf{H}(E[j])\in k\text{-}\mathsf{mod}\) for any \(j \in \mathbb{Z}\). Fix a finite set of generators \(J_j = \{\varphi_{js}\}\) for \(\mathbf{H}(E[j])\). By Yoneda's lemma, each \(\varphi_{js}\) corresponds to a morphism \(\varphi_{js}: {\Hom}_\T(E[j], -) \to \mathbf{H}\), which canonically induces an epimorphism
		\[\bigoplus_{j \in \mathbb{Z}}{\Hom}_\T(\bigoplus_{s \in J_j}E[j], -)|_{\langle E\rangle_1}\simeq
		\bigoplus_{j \in \mathbb{Z}} \bigoplus_{s \in J_j} {\Hom}_\T(E[j], -)|_{\langle E\rangle_{1}} \to \mathbf{H}_{\langle E\rangle_{1}}.
		\]We now prove the isomorphism
		\[
		\bigoplus_{j \in \mathbb{Z}}{\Hom}_\T(\bigoplus_{s \in J_j}E[j], -)|_{\langle E\rangle_1}\simeq {\Hom}_\T( \bigoplus_{j \in \mathbb{Z}} \bigoplus_{s \in J_j} E[j], - )|_{\langle E\rangle_1}.
		\]
		Observe that
		\[
		{\Hom}_\T( \bigoplus_{j \in \mathbb{Z}} \bigoplus_{s \in J_j} E[j], - )|_{\langle E\rangle_1} \simeq \prod_{j \in \mathbb{Z}} {\Hom}_\T( \bigoplus_{s \in J_j} E[j], - )|_{\langle E\rangle_1},
		\]
		which induces a canonical morphism
		\[
		\psi: \bigoplus_{j \in \mathbb{Z}}{\Hom}_\T(\bigoplus_{s \in J_j}E[j], -)|_{\langle E\rangle_1} \to \prod_{j \in \mathbb{Z}} {\Hom}_\T( \bigoplus_{s \in J_j} E[j], - )|_{\langle E\rangle_1}.
		\]
		Thus it suffices to verify that \(\psi\) is an isomorphism. This reduces to showing that \(\psi(E[r])\) is an isomorphism for each \(r \in \mathbb{Z}\). Note that
		\[
		{\Hom}_\T( \bigoplus_{s \in J_j} E[j], E[r] ) \simeq \mathbf{D}{\Hom}_\T(G[r], \bigoplus_{s \in J_j} E[j]).
		\]
		By the construction of \(J_j\), we have \({\Hom}_\T( \bigoplus_{s \in J_j} E[j], - )_{\langle E\rangle_1} = 0\) for \(j \gg 0\). Since \(G[r] \in \T^{-}\) and \(E \in \T^{+}\), it follows that
		\[
		\mathbf{D}{\Hom}_\T(G[r], \bigoplus_{s \in J_j} E[j]) = 0 \quad \text{for} \quad j \ll 0.
		\]
		Thus \({\Hom}_\T( \bigoplus_{s \in J_j} E[j], E[r]) = 0\) for \(j \ll 0\). Clearly, \(\Hom_{\T}(\bigoplus_{s\in J_j} E[j], E[r])=0\) for \(|j|\gg 0\) implies that \(\psi(E[r])\) is an isomorphism. Let \( F_1 = \bigoplus_{j \in \mathbb{Z}} \bigoplus_{s \in J_j} E[j] \). We now need to construct a strong \(\langle E\rangle_1\)-coapproximating system for \(F_1\). Since \(E \in \T^{+}\), there is an integer \(B < 0\) such that \(E[B] \in \T^{\geq 1}\). For each \(m > 0\), define
		\[
		E_m = \bigoplus_{j \geq -m + 2 + B} \bigoplus_{s \in J_j} E[j].
		\]
		The sum is finite by hypothesis, so \(E_m \in \langle E \rangle_1\). Clearly, there is a canonical split epimorphism from \(E_{m+1}\) to \(E_m\). Moreover, there exists a canonical split epimorphism from \(F_1\) to \(E_m\). It is straightforward to verify that this construction yields a strong \(\langle E \rangle_1\)-coapproximating system for \(F_1\). By Lemma~\ref{universal property of coapproximating system}(2), we obtain $F_{1}\in\T_{c}^{+}$.
		
		Suppose we have established the case $n=r$. Now consider $n=r+1$, so that $\mathbf{H}$ is a locally finite $\langle E\rangle_{r+1}$-homological functor. Clearly, $\mathbf{H}|_{\langle E\rangle_r}$ is a locally finite $\langle E\rangle_{r}$-homological functor. Thus, there exist $F_{r} \in \T_{c}^{+}$ and an epimorphism $\varphi_{r}: {\Hom}_\T(F_{r},-)|_{\langle E\rangle_{r}} \to \mathbf{H}|_{\langle E\rangle_r}$. Let $\mathbf{H}'$ denote the kernel of $\varphi_{r}$. By Lemma \ref{reasonable} and Lemma \ref{ker of homo functor}, $\mathbf{H}'$ is also a locally finite $\langle E\rangle_{r}$-homological functor. Thus we obtain $F' \in \T_{c}^{+}$ and an epimorphism $\varphi' : {\Hom}_\T(F',-)|_{\langle E\rangle_{1}} \to \mathbf{H}'|_{\langle E\rangle_{1}}$. 
		Moreover, $F'$ admits a strong $\langle E\rangle_1$-coapproximating system: 
		\[
		\cdots \longrightarrow E'_{3}\longrightarrow E'_{2} \longrightarrow E'_{1}
		\]
		where we can assume all morphisms in this sequence are split epimorphisms. Denote the composition 
		\[
		{\Hom}_\T(F',-)|_{\langle E\rangle_{1}} \longrightarrow \mathbf{H}'|_{\langle E\rangle_{1}} \longrightarrow {\Hom}_\T(F_{r},-)|_{\langle E\rangle_{1}}
		\] 
		by $\tilde{\varphi}$. Then \cite[Lemma~7.8]{N18a} implies there exists $g_{r}: F_{r} \to F'$ with $\T(g_{r},-)|_{\langle E\rangle_{1}} = \tilde{\varphi}$. 
		By Lemma \ref{extend weak triangle}, there exists a sequence in $\T_{c}^{+}$:
		\[
		F'[-1] \longrightarrow F_{r+1} \stackrel{{f_{r}}}\longrightarrow F_{r} \stackrel{{g_{r}}}\longrightarrow F'
		\]
		 such that{\xiaowuhao
		\[
		{\Hom}_{\T}(F',-)|_{\E} \xrightarrow{\T(g_{r},-)|_{\E}} {\Hom}_{\T}(F_{r},-)|_{\E} \xrightarrow{\T(f_{r},-)|_{\E}} {\Hom}_{\T}(F_{r+1},-)|_{\E} \to {\Hom}_{\T}(F'[-1],-)|_{\E}
		\]}
		forms a weak triangle in $\mathsf{Hom}_{k}(\mathcal{E}, k\text{-}\mathsf{Mod})$, and $F_{r+1}$ admits a strong $\langle E\rangle_{r+1}$-coapproximating system. 
		
		Setting $\mathcal{A} = \langle E\rangle_{1}$, $\mathcal{B} = \langle E\rangle_{r}$, $\mathcal{C} = \langle E\rangle_{r+1}$, \cite[Lemma~8.5]{N18a} yields a morphism 
		\[
		\varphi_{r+1}: {\Hom}_\T(F_{r+1},-)|_{\langle E\rangle_{r+1}} \longrightarrow \mathbf{H}
		\]
		making the following diagram 
		\[\begin{tikzcd}
			{{\Hom}_\T(F_{r},-)|_{\langle E\rangle_{r}}} \arrow[rrrr, "\varphi_{r}"', bend right] \arrow[rr, "{\T(f_{r},-)|_{\langle E\rangle_{r}}}"] && {{\Hom}_\T(F_{r+1},-)|_{\langle E\rangle_{r}}} \arrow[rr, "\varphi_{r+1}|_{\langle E\rangle_{r}}"] && \mathbf{H}|_{\langle E\rangle_{r}}.
		\end{tikzcd}\]
		commutes. 
		Note further that the sequence
		\[
		{\Hom}_\T(F',-)|_{\langle E\rangle_{1}} \xrightarrow{\T(g_r,-)|_{\langle E\rangle_{1}}} {\Hom}_\T(F_r,-)|_{\langle E\rangle_{1}} \xrightarrow{\varphi_{r} |_{\langle E\rangle_{1}}}\mathbf{H}|_{\langle E\rangle_{1}} \to 0
		\]
		is exact. Thus by \cite[Lemma~8.6]{N18a}, 
		$$\varphi_{r+1}: {\Hom}_\T(F_{r+1},-)|_{\langle E\rangle_{r+1}} \longrightarrow \mathbf{H}$$
		is an epimorphism. 
		Finally, \cite[Lemma~8.4]{N18a} implies that the sequence
		\[
		{\Hom}_\T(F',-)|_{\langle E\rangle_{1}} \xrightarrow{\T(g_r,-)|_{\langle E\rangle_{1}}} {\Hom}_\T(F_r,-)|_{\langle E\rangle_{1}} \xrightarrow{\T(f_r,-)|_{\langle E\rangle_{1}}} {\Hom}_\T(F_{r+1},-)|_{\langle E\rangle_{1}}
		\]
		is exact, which completes the proof of (3).
	\end{proof}
	
	\begin{rem}\label{lifting1}
		Let $\mathbf{H}$ be an $\E$-homological functor and $F_n \in \T_c^+$ with a strong $\langle E\rangle_{n}$-coapproximating system. Then by \cite[Corollary~7.4]{N18a}, each $\tilde{\varphi}_n \colon {\Hom}_\T(F_n, -)|_{\langle E \rangle_n} \to \mathbf{H}|_{\langle E \rangle_n}$ admits a lift $\varphi_n \colon {\Hom}_\T(F_n, -)|_{\E} \to \mathbf{H}$ such that $\varphi_n|_{\langle E \rangle_n} = \tilde{\varphi}_n$.
	\end{rem}
	
	\begin{prop}\label{brutal rep}
		Let $\mathbf{H}$ be a locally finite $\E$-homological functor. Then there exist $F \in \T$ and an isomorphism $\varphi: {\Hom}_\T(F, -)|_{\E} \simeq \mathbf{H}$.
	\end{prop}
	\begin{proof}
		Since for every $n>0$, the restriction $\mathbf{H}|_{\langle E\rangle_{n}}$ is a locally finite $\langle E\rangle_{n}$-homological functor, Lemma \ref{tech5} yields a sequence 
		\[(F_{*},f_{*}):\quad
		\cdots \longrightarrow F_{3} \stackrel{f_{2}}\longrightarrow F_{2} \stackrel{f_{1}}\longrightarrow F_{1}
		\] 
		in $\T_{c}^{+}$ and morphisms $\tilde{\varphi}_{n}: {\Hom}_\T(F_{n},-)|_{\langle E\rangle_{n}} \to \mathbf{H}|_{\langle E\rangle_{n}}$. By Remark \ref{lifting1}, each $\tilde{\varphi}_{n}$ admits a lift $\varphi_{n}: {\Hom}_\T(F_{n},-)|_{\E} \to \mathbf{H}$ satisfying $\varphi_{n}|_{\langle E\rangle_{n}} = \tilde{\varphi}_{n}$. Then \cite[Corollary~7.4]{N18a} implies that for every $i$, the following diagram 
		\[\begin{tikzcd}
			{{\Hom}_\T(F_{i},-)|_{\E}} \arrow[rd, "\varphi_{i}"'] \arrow[rr, "{\T(f_{i},-)|_{\E}}"] &            & {{\Hom}_\T(F_{i+1},-)|_{\E}} \arrow[ld, "\varphi_{i+1}"] \\
			& \mathbf{H} &                                                     
		\end{tikzcd}.\] commutes. This yields a canonical morphism 
		$$\varphi: \colim \Hom_\T(F_{\ast}, -)|_{\E} \to \mathbf{H}.$$ 
		We first prove that $\varphi$ is an isomorphism, and then establish that 
		$$\colim \Hom_\T(F_{\ast}, -)|_{\E} \simeq {\Hom}_\T(\Holim F_{\ast}, -)|_{\E}.$$
		
		To prove that $\varphi$ is an isomorphism, it suffices to show that $\varphi|_{\langle E\rangle_{1}}$ is an isomorphism. Define $K_{n} = \ker(\varphi_{n}|_{\langle E\rangle_{1}})$; then by Lemma \ref{tech5}(3), we obtain the following commutative diagram:
		\[\begin{tikzcd}
0 \arrow[r] & K_{n} \arrow[r] \arrow[d, "0"'] & {{\Hom}_\T(F_{n},-)|_{\langle E\rangle_{1}}} \arrow[r, "\varphi_{n}|_{\langle E\rangle_{1}}"] \arrow[d, "{\T(f_{n},-)|_{\langle E\rangle_{1}}}"'] & \mathbf{H}|_{\langle E\rangle_{1}} \arrow[r] \arrow[d, no head, shift right] \arrow[d, no head] & 0 \\
0 \arrow[r] & K_{n+1} \arrow[r]               & {{\Hom}_\T(F_{n+1},-)|_{\langle E\rangle_{1}}} \arrow[r, "\varphi_{n+1}|_{\langle E\rangle_{1}}"']                                                & \mathbf{H}|_{\langle E\rangle_{1}} \arrow[r]                                                    & 0
\end{tikzcd}.\] 
		By the right exactness of the colimit functor, we obtain an exact sequence
		\[
		0=\colim K_{\ast} \longrightarrow \colim \Hom_\T(F_{\ast}, -)|_{\langle E\rangle_{1}} \xrightarrow{\varphi|_{\langle E\rangle_{1}}} \mathbf{H}|_{\langle E\rangle_{1}} \longrightarrow 0.
		\] Thus $\varphi|_{\langle E\rangle_{1}}$ is an isomorphism. 
		
		We now prove that $\colim \Hom_\T(F_{\ast}, -)|_{\E} \simeq {\Hom}_\T(\Holim(F_{*}), -)|_{\E}$. It suffices to prove that for any $j \in \mathbb{Z}$, ${\Hom}_\T(\Holim(F_{*}), E[j])$ is the colimit in $k\text{-}\mathsf{Mod}$ of the sequence ${\Hom}_\T(F_{*}, E[j])$. Note that $${\Hom}_\T(\Holim(F_{*}), E[j]) \simeq \mathbf{D}{\Hom}_\T(G[j], \Holim(F_{*})).$$ 
		By Lemma \ref{homfinitelem}(1), each ${\Hom}_\T(G[j], F_{i})$ is finite-dimensional. Thus, Lemma \ref{ML lem2} gives 
		\[{\Hom}_\T(G[j], \Holim(F_{*})) \simeq \lims\Hom_\T(G[j], F_{*}).\]
		If we can show ${\Hom}_\T(G[j], \Holim(F_{*}))$ is finite-dimensional, then $\mathbf{D}{\Hom}_{\T}(G[j], \Holim(F_{*}))$ is the colimit of the sequence $\mathbf{D}{\Hom}_{\T}(G[j],F_{*}))$ in $k\text{-}\mathsf{mod}$. We will prove the following claim:

    Let $(V_{*}, f_{*})$ be the sequence $V_1 \xrightarrow{f_1} V_2 \xrightarrow{f_2} V_3 \to \cdots$ in $k\text{-}\mathsf{mod}$. If $V$ is the colimit of $(V_{*}, f_{*})$ in $k\text{-}\mathsf{mod}$, then $V$ is also its colimit in $k\text{-}\mathsf{Mod}$.

    The desired result follows directly from the claim. We now prove the claim. Let $\gamma_{i}: V_{i} \to V$ be the canonical map. Without loss of generality, we may assume that each $\gamma_{i}$ is surjective for all $i > 1$. Fix \( W \in k\text{-}\mathsf{Mod} \).  
	By \cite[Corollary~3.9]{Rotman09}, we may write  
	\( W = \bigoplus_{j \in J} W_{j} \) with each \( W_{j} \in k\text{-}\mathsf{mod} \).  
	Let \( \sigma_{j} \colon W_{j} \to \bigoplus_{j \in J} W_{j} \) be the canonical injection.  
	By the universal property of the coproduct, for each \( r \in J \) there exists a morphism  
	\( \pi_{r} \colon \bigoplus_{j \in J} W_{j} \to W_{r} \) such that  
	\( \pi_{r} \sigma_{r} = \mathrm{Id}_{W_{r}} \) and  
	\( \pi_{r} \sigma_{j} = 0 \) for \( j \neq r \).
	
	Consider a morphism of sequences \( h_{*} \colon V_{*} \to W \), and assume that  
	\( h_{i'} \neq 0 \) for some \( i' \ge 1 \).  
	For each \( j \in J \), the composition  
	\( \pi_{j} h_{i} \colon V_{i} \to W_{j} \) defines a morphism of sequences  
	\( V_{*} \to W_{j} \).  
	Since \( V \) is the colimit of \( (V_{*}, f_{*}) \) in \( k\text{-}\mathsf{mod} \),  
	there exists a unique morphism \( g_{j} \colon V \to W_{j} \) such that  
	\[
	g_{j} \gamma_{i} = \pi_{j} h_{i} \qquad \text{for all } i \ge 1.
	\]
	Because each \( \gamma_{i} \) is surjective, we have  
	\( \pi_{j} h_{i} \neq 0 \) for some (equivalently, every) \( i \ge 1 \)  
	if and only if \( g_{j} \neq 0 \).
	
	We now show that only finitely many \( j \in J \) satisfy \( g_{j} \neq 0 \).  
	Since \( V_{1} \) is compact in \( k\text{-}\mathsf{Mod} \),  
	the morphism \( h_{1} \colon V_{1} \to W \) factors through a finite direct sum,  
	yielding a commutative diagram
	\[
	\begin{tikzcd}
		V_{1} \arrow[rr, "h_{1}"] \arrow[rd] 
		& & W \\
		& \bigoplus_{i=1}^{n} W_{1i} \arrow[ru, hook] &
	\end{tikzcd}.
	\]
	
	For each \( j \in J \), we have \( g_{j} \neq 0 \)  
	if and only if \( \pi_{j} h_{1} \neq 0 \).  
	By construction of the projections \( \pi_{j} \),  
	this implies that the set \( \{ j \in J : g_{j} \neq 0 \} \) is finite.  
	Without loss of generality, let  
	\( \{ t_{1}, \dots, t_{m} \} \subseteq J \)  
	be precisely the indices with \( g_{t_{s}} \neq 0 \).  
	This yields a unique morphism  
	\[
	V \longrightarrow \bigoplus_{s=1}^{m} W_{t_{s}}.
	\]
	Composing with the canonical inclusion  
	\( \bigoplus_{s=1}^{m} W_{t_{s}} \hookrightarrow W \),  
	we obtain a morphism \( g \colon V \to W \).  
	It is straightforward to verify that  
	\( g \gamma_{i} = h_{i} \) for all \( i \ge 1 \).  
    The uniqueness of \( g \) follows from the surjectivity of \( \gamma_{i} \), completing the proof of the claim.
		
	Now it suffices to prove that ${\Hom}_\T(G[j], \Holim(F_{*}))$ is finite-dimensional. Since \[\colim \Hom_\T(F_{*}, E[j]) \simeq \mathbf{H}(E[j])\]and $\mathbf{H}(E[j])$ is finite-dimensional, \( \colim \Hom_{\T}(F_{*}, E[j]) \) is finite-dimensional. 
		For all $i>0$, we have 
		$${\Hom}_\T(F_i, E[j]) \simeq\mathbf{D}{\Hom}_\T(G[j], F_i),$$
		so the colimit of the sequence $\mathbf{D}{\Hom}_{\T}(G[j], F_{*})$ is finite-dimensional. By definition, $\mathbf{D}(\colim \Hom_{\T}(F_{*}, E[j]))$ is finite-dimensional and coincides with the limit of the sequence $\mathbf{D}\mathbf{D}{\Hom}_{\T}(G[j], F_{*})$; hence the limit is also finite-dimensional. Since $\mathbf{D}$ is the duality functor on $k\text{-}\mathsf{mod}$, the limit of the sequences $\mathbf{D}\mathbf{D}{\Hom}_{\T}(G[j], F_{*})$ and ${\Hom}_\T(G[j], F_{*})$  coincide. Consequently, 
		$${\Hom}_\T(G[j], \Holim(F_{*}))\simeq \lims \Hom_\T(G[j], F_{*})$$
		is finite-dimensional, 
		completing the proof.
	\end{proof}
	
	\subsection{Formal representability theorems}
	The representability result established in Proposition~\ref{brutal rep} shows that every locally finite 
	$\E$-homological functor is representable by an object of $\T$.  
	In order to refine this statement and identify the representing objects inside the intrinsic subcategories 
	$\T^+_c$ and $\T^b_c$, we now restrict our attention to triangulated categories equipped with compact silting objects.  
	Such categories enjoy additional structural properties that allow us to control the cohomological degrees of the representing objects. This dual representability framework will serve as a foundation for the applications developed in the next section.
	
	Let $\S$ be a triangulated category with coproducts. An object $S$ in $\S$ is called a {\em silting object} if the pair $(S^{\perp_{>0}},S^{\perp_{<0}})$ of subcategories forms a $t$-structure on $\S$, where \begin{align*}
    S^{\perp_{>0}} &\coloneqq \{M\in\S \mid \Hom_{\S}(S,M[i])=0\ \text{for}\ i>0\},\\
     S^{\perp_{<0}} &\coloneqq \{M\in\S \mid \Hom_{\S}(S,M[i])=0\ \text{for}\ i<0\}.
    \end{align*}Clearly, $S^{\perp_{>0}}$ is closed under products. Moreover, if $S$ is compact, we call it a {\em compact silting object}. In this case, $S^{\perp_{>0}}=\S_S^{\leq 0}$ and $S^{\perp_{<0}}=\S_S^{\geq 0}$.
	
	In the remainder of this section, we assume that $\T$ is a locally Hom-finite $k$-linear triangulated category with a compact silting object $G$. Note that in this case we have $\Hom_{\T}(G,G[i])=0$ for $i>0$. Then by \cite[Remark~5.3]{N18a}, $\T$ is approximable. This setting includes many important examples and provides the necessary control over cohomological degrees 
	to obtain refined representability theorems.
	
	\begin{exam}
	Let $A$ be a non-positive DG $k$-algebra with each $H^i(A)$ finite-dimensional. Then the derived category $\D(A)$ is locally Hom-finite and admits a compact silting object, namely $A$ itself.
	\end{exam}
	
	\begin{lem}\label{lem:key2}
		$\T^+_c\cap \T^{\geq 0}\subseteq \T^{\geq 1}*{\rm add}(E)\subseteq \T^{\geq 1}*\langle E\rangle_1^{[-1,1]}. $
	\end{lem}
	
\begin{proof}
	Let  
	\( T \in \T_{c}^{+} \cap \T^{\ge 0} \). By Lemma~\ref{homfinitelem}, the space \( {\Hom}_{\T}(T,E) \) is finite-dimensional.  
	Choose a basis \( \{f_{1},\dots,f_{m}\} \subseteq {\Hom}_{\T}(T,E) \), and let  
	\( f \colon T \to E^{m} \) be the morphism induced by these maps. Note that $E\in\T^{\geq 0}$. This gives a triangle
	\[
	E^{m}[-1] \longrightarrow D \longrightarrow T \stackrel{f}{\longrightarrow} E^{m}
	\]
with \( D \in \T^{\ge 0} \) and \( E^{m} \in \mathrm{add}(E) \).  
	It suffices to show that \( D \in \T^{\ge 1} \).
	
	Since \( G \) is silting, we have  
	\( {\Hom}_{\T}(E[-j],E) = 0 \) for all \( j>0 \).  
	Hence every morphism in \( {\Hom}_{\T}(D,E) \) factors through \( T \).  
	Because \( f \) is an \( \langle E\rangle \)-approximation, it follows that  
	\( {\Hom}_{\T}(D,E) = 0 \), and therefore \( {\Hom}_{\T}(G,D) = 0 \).
	
	Moreover, since \( D \in \T^{\ge 0} \) and \( G\in\T^{\leq 0}\), we also have  
	\( {\Hom}_{\T}(G[i],D) = 0 \) for all \( i>0 \).  
	Thus \( D \in \T^{\ge 1} \), completing the proof.
\end{proof}

	Note that for any $n \in \mathbb{Z}$ and $T \in \T_{c}^{+} \cap \T^{\geq n}$, we have $T[n] \in \T_{c}^{+} \cap \T^{\geq 0}$. This implies the following corollary.
	
	\begin{cor}\label{cor:key2}
		For any $n \in \mathbb{Z}, m > 0$, 
		\[
		\T_{c}^{+} \cap \T^{\geq n} \subseteq 
		\T^{\geq n+m} \ast \langle E\rangle_{m}^{[n-1,n+m]}.
		\]
	\end{cor}
\begin{proof} The assertion follows immediately by combining the octahedron axiom with Lemma~\ref{lem:key2}. We leave the details to the reader.
\end{proof}
	
	\begin{lem}\label{B.1}
	Let $\mathbf{H}$ be a locally finite $\E$-homological functor. Then there exists an integer $A > 0$ such that for every $n > 0$, there exists $K_{n} \in \T_c^{+} \cap \T^{\geq -A}$ and a morphism $\psi_{n}: {\Hom}_{\T}(K_{n}, -)|_{\E} \to \mathbf{H}$ such that $\psi_{n}|_{\langle E \rangle_n}$ is an epimorphism.
	\end{lem}
\begin{proof}
	Since \( \mathbf{H} \) is locally finite, there exists \( A>0 \) such that  
	\( \mathbf{H}(E[i]) = 0 \) for all \( i \ge A-1 \).  
	As \( \mathbf{H} \) is \( \E \)-homological, it follows that  
	\( \mathbf{H}(E') = 0 \) for all  
	\( E' \in \langle E\rangle^{(-\infty, -A+1]} \).
	
	Fix \( n>0 \).  
	By Lemma~\ref{tech5} and Remark~\ref{lifting1}, there exist  
	\( F_{n} \in \T_{c}^{+} \) and a morphism  
	$$\varphi_{n} \colon \Hom_{\T}(F_{n}, -)|_{\E} \longrightarrow \mathbf{H} $$
	such that \( \varphi_{n}|_{\langle E\rangle_{n}} \) is an epimorphism.  
	Since \( \T_{c}^{+} \subseteq \T^{+} \), there exists \( \ell \in \mathbb{Z} \) with  
	\( F_{n} \in \T^{\ge \ell} \).  
	Using the fact that \( \T^{\ge 0}[-1] \subseteq \T^{\ge 0} \), we may assume  
	\( -\ell - A + 1 > 0 \) and set \( m = -\ell - A + 1 \).  
	By Corollary~\ref{cor:key2}, we obtain a triangle
	\[
	D_{m} \longrightarrow F_{n} \stackrel{\alpha}{\longrightarrow} E_{m} 
	\longrightarrow D_{m}[1]
	\]
with  
	\( D_{m} \in \T^{\ge -A+1} \) and  
	\( E_{m} \in \langle E\rangle^{[\ell-1, -A+1]}_{-\ell-A+1} 
	\subseteq \langle E\rangle^{(-\infty, -A+1]} \).
	The composition
	\[
	{\Hom}_{\T}(E_{m}, -)|_{\E}
	\xrightarrow{\T(\alpha, -)|_{\E}}
	{\Hom}_{\T}(F_{n}, -)|_{\E}
	\xrightarrow{\varphi_{n}}
	\mathbf{H}
	\]
	is zero.  
	Moreover, \( E_{m} \) admits a strong \( \E \)-coapproximating system
	\[
	\cdots \to E_{m} \xrightarrow{\mathrm{Id}} E_{m} 
	\xrightarrow{\mathrm{Id}} E_{m}.
	\]
	By Lemma~\ref{extend weak triangle}(1), there exists a sequence
	\[
	E_{m}[-1] \longrightarrow K_{n} \xrightarrow{\beta} F_{n} 
	\xrightarrow{\alpha} E_{m}
	\]
	such that
	\[
	{\Hom}_{\T}(E_{m}, -)|_{\E}
	\to {\Hom}_{\T}(F_{n}, -)|_{\E}
	\to {\Hom}_{\T}(K_{n}, -)|_{\E}
	\to {\Hom}_{\T}(E_{m}[-1], -)|_{\E}
	\]
	is a weak triangle in  
	\( \mathsf{Hom}_{k}(\E, k\text{-}\mathsf{Mod}) \).  
	By \cite[Lemma~8.5]{N18a}, there exists a morphism  
	\( \psi_{n} \colon {\Hom}_{\T}(K_{n}, -)|_{\E} \to \mathbf{H} \)  
	and a commutative diagram
	\[
	\begin{tikzcd}
		{\Hom}_{\T}(F_{n}, -)|_{\E} 
		\arrow[r, "{\T(\beta, -)|_{\E}}"] 
		\arrow[rr, "\varphi_{n}"', bend right] &
		{\Hom}_{\T}(K_{n}, -)|_{\E} 
		\arrow[r, "\psi_{n}"] &
		\mathbf{H}.
	\end{tikzcd}
	\]
	Since \( \varphi_{n}|_{\langle E\rangle_{n}} \) is an epimorphism,  
	so is \( \psi_{n}|_{\langle E\rangle_{n}} \).  
	As \( A \) is fixed, it remains to show that  
	\( K_{n} \in \T^{\ge -A} \).
	
	From the triangle
	\[
	D_{m} \to F_{n} \xrightarrow{\alpha} E_{m} \to D_{m}[1],
	\]
	we obtain an exact sequence
	\[
	\mathbf{H}^{i}(D_{m}) \to \mathbf{H}^{i}(F_{n})
	\xrightarrow{\mathbf{H}^{i}(\alpha)}
	\mathbf{H}^{i}(E_{m}) \to \mathbf{H}^{i+1}(D_{m}).
	\]
	Since \( D_{m} \in \T^{\ge -A+1} \),  
	\( \mathbf{H}^{i}(\alpha) \) is an isomorphism for all \( i \le -A-1 \).  
	By Lemma~\ref{extend weak triangle}(2), we have an exact sequence
	\[
	\mathbf{H}^{i-1}(E_{m}) \to \mathbf{H}^{i}(K_{n}) 
	\to \mathbf{H}^{i}(F_{n})
	\xrightarrow{\mathbf{H}^{i}(\alpha)}
	\mathbf{H}^{i}(E_{m}).
	\]
	Hence \( \mathbf{H}^{i}(K_{n}) = 0 \) for all \( i \le -A-1 \).  
	Lemma~\ref{tech1} then implies  
	\( K_{n} \in \T^{\ge -A} \), completing the proof.
\end{proof}

	\begin{lem}\label{B.2}
		Let $\mathbf{H}$ be a locally finite $\E$-homological functor and $A\geq 0$ be an integer. Suppose $F_{1}$, $F_{2}$ are objects in $\T_{c}^{+}\cap\T^{\geq -A}$, $E'$ is an object in $\E\cap\T^{\geq -A}$, and $\alpha:F_{1}\to E'$ is a morphism. Assume we are given an integer $m > 0$ and natural transformations 
		$$\varphi_1: {\Hom}_\mathcal{T}(F_1, -)|_{\E} \to \mathbf{H}~\text{and}~\varphi_2: {\Hom}_\mathcal{T}(F_2, -)|_{\E} \to \mathbf{H}$$
		such that $\varphi_1|_{\langle E \rangle_m}$ is an epimorphism. Then there exists a commutative diagram
		\[\begin{tikzcd}
			& {{\Hom}_{\T}(\widetilde{E},-)|_{\E}} \arrow[r, "{\mathcal{T}(\widetilde{\alpha},-)|_{\E}}"] \arrow[rdd, "{\mathcal{T}(\gamma,-)|_{\E}}"'] & {{\Hom}_{\T}(\widetilde{F},-)|_{\E}} \arrow[rdd, "\widetilde{\varphi}"'] & {{\Hom}_{\T}(F_{2},-)|_{\E}} \arrow[l, "{\mathcal{T}(\beta,-)|_{\E}}"'] \arrow[dd, "\varphi_{2}"] \\
			&                                                                                                          &                                                                        &                                                                     \\
			{{\Hom}_{\T}(E',-)|_{\E}} \arrow[rr, "{\mathcal{T}(\alpha,-)|_{\E}}"'] \arrow[ruu, "{\mathcal{T}(\delta,-)|_{\E}}"] &                                                                                                                                         & {{\Hom}_{\T}(F_{1},-)|_{\E}} \arrow[r, "\varphi_{1}"']                   & \mathbf{H}                                                                                     
		\end{tikzcd}\]such that
		
		{\rm (1)} $\widetilde{E}\in\E\cap\T^{\geq -A}$ and $\widetilde{F}\in\T_{c}^{+}\cap\T^{\geq -A}$.
		
		{\rm (2)} $\mathbf{H}^{i}(\widetilde{\alpha})$ and $\mathbf{H}^{i}(\gamma)$ are isomorphisms for all $i\leq m-2$.
	\end{lem}
	\begin{proof}
		Without loss of generality, assume \( A = 0 \). The proof is divided into three steps.
		
		{\bf Step 1:} Construct $\widetilde{E}\in\E\cap\T^{\geq 0}$, $\delta:\widetilde{E}\to E'$ and $\gamma:F_{1}\to \widetilde{E}$.
		
		By Corollary \ref{cor:key2}, there exists a triangle $$D_{m}\stackrel{a}\longrightarrow F_{2}\stackrel{b}\longrightarrow E_{m}'\longrightarrow D_{m}[1]$$ 
		with $D_{m}\in\T^{\geq m}$ and $E_{m}'\in\langle E\rangle_{m}^{[-1,m]}$. Since $E_{m}' \in \langle E \rangle_{m}$ and $\varphi_{1}|_{\langle E \rangle_{m}}$ is an epimorphism, there exists a morphism $f: F_{1} \to E_{m}'$ such that the following diagram commutes:\[\begin{tikzcd}
			& {{\Hom}_\T(E_{m}',-)|_{\langle E\rangle_{m}}} \arrow[d, "{\T(b,-)|_{\langle E\rangle_{m}}}"] \arrow[ldd, "{\T(f,-)|_{\langle E\rangle_{m}}}"'] \\
			& {{\Hom}_\T(F_{2},-)|_{\langle E\rangle_{m}}} \arrow[d, "\varphi_{2}|_{\langle E\rangle_{m}}"]                                                  \\
			{{\Hom}_\T(F_{1},-)|_{\langle E\rangle_{m}}} \arrow[r, "\varphi_{1}|_{\langle E\rangle_{m}}"'] & \mathbf{H}|_{\langle E\rangle_{m}}.         
		\end{tikzcd}\]Combining \cite[Lemma~7.8]{N18a} with \cite[Corollary~7.4]{N18a}, we also have $$\varphi_{1}\circ\T(f,-)|_{\E}=\varphi_{2}\circ\T(b,-)|_{\E}.$$
		Note that $F_{1}\in\T_{c}^{+}$. By Proposition \ref{Tc+coapproximable}, $F_{1}$ admits a strong $\E$-coapproximating system:\[\cdots \stackrel{f_3}\longrightarrow E_3 \stackrel{f_2}\longrightarrow E_2 \stackrel{f_1}\longrightarrow E_1.\]

		Then by Corollary~\ref{key1}, the morphism \scalebox{0.6}{$\begin{pmatrix} \alpha \\ f \end{pmatrix}$}~$\in {\Hom}_{\T}(F_1, E' \oplus E_m')$ factors through some $E_j$. Choosing such an $E_j$ with $j \geq m$ and setting $\widetilde{E} \coloneqq E_j$, we obtain commutative diagrams\[\begin{tikzcd}
			F_{1} \arrow[rr, "\scalebox{0.6}{$\begin{pmatrix} \alpha \\ f \end{pmatrix}$}"] \arrow[rd, "\gamma"'] &                                                             & E'\oplus E_{m}' \arrow[r, "{(1,0)}"] & E' \\
			& \widetilde{E} \arrow[ru] \arrow[rru, "\delta"', bend right] &                                      &   
		\end{tikzcd}\]and\[\begin{tikzcd}
			F_{1} \arrow[rr, "\scalebox{0.6}{$\begin{pmatrix} \alpha \\ f \end{pmatrix}$}"] \arrow[rd, "\gamma"'] &                                                             & E'\oplus E_{m}' \arrow[r, "{(0,1)}"] & E_{m}' \\
			& \widetilde{E} \arrow[ru] \arrow[rru, "\delta'"', bend right] &                                      &   
		\end{tikzcd}.\]By the definition of strong $\E$-coapproximating system, $\mathbf{H}^{i}(\ga):\mathbf{H}^{i}(F_{1})\to \mathbf{H}^{i}(\widetilde{E})$ is an isomorphism for all $i\leq m$. Since $F_{1}\in\T_{c}^{+}\cap\T^{\geq 0}$, we have $\mathbf{H}^{i}(\widetilde{E})=0$ when $i<0$. Then by Lemma \ref{tech1}, we obtain $\widetilde{E}\in\E\cap\T^{\geq 0}$. The information we have obtained so far yields the following commutative diagram ($*$)
		
		\[\begin{tikzcd}
			& {{\Hom}_\T(\widetilde{E},-)|_{\E}} \arrow[rdd, "{\T(\gamma,-)|_{\E}}"'] &                                               & {{\Hom}_\T(E_{m}',-)|_{\E}} \arrow[d, "{\T(b,-)|_{\E}}"] \arrow[ll, "{\T(\delta',-)|_{\E}}"'] \\
			&                                                                  &                                               & {{\Hom}_\T(F_{2},-)|_{\E}} \arrow[d, "\varphi_{2}"]                                           \\
			{{\Hom}_\T(E',-)|_{\E}} \arrow[rr, "{\T(\alpha,-)|_{\E}}"'] \arrow[ruu, "{\T(\delta,-)|_{\E}}"] &                                                                  & {{\Hom}_\T(F_{1},-)|_{\E}} \arrow[r, "\varphi_{1}"'] & \mathbf{H}.                                  
		\end{tikzcd}\]
		
		{\bf Step 2:} Construct $\widetilde{F}\in\T_{c}^{+}\cap\T^{\geq 0}$, $\widetilde{\alpha}:\widetilde{F}\to \widetilde{E}$ and $\beta:\widetilde{F}\to F_{2}$.
		
		By $(\ast)$, there is a morphism $(-\delta',b):\widetilde{E}\oplus F_{2}\to E_{m}'$. Then by Lemma \ref{extend weak triangle}(1), we have the sequence $$E_{m}'[-1]\xrightarrow{\sigma}\widetilde{F}\xrightarrow{\scalebox{0.6}{$\begin{pmatrix} \widetilde{\alpha} \\ \beta \end{pmatrix}$}}\widetilde{E}\oplus F_{2}\xrightarrow{(-\delta',b)}E_{m}'$$
		in $\T_{c}^{+}$ such that
		{\xiaowuhao
			\[(E_{m}',-)|_{\E}\xrightarrow{\T((-\delta',b),-)|_{\E}}(\widetilde{E}\oplus F_{2},-)|_{\E}\xrightarrow{\T(\scalebox{0.6}{$\begin{pmatrix} \widetilde{\alpha} \\ \beta \end{pmatrix}$},-)|_{\E}}(\widetilde{F},-)|_{\E}\xrightarrow{\T(\sigma,-)|_{\E}}(E_{m}'[-1],-)|_{\E}\]}
		is a weak triangle. This implies $\delta'\widetilde{\alpha}=b\beta$. Therefore, we obtain the following commutative diagram ($\ast\ast$)
		\[\begin{tikzcd}
			&                                                                                                                                          &                                                      & {{\Hom}_{\T}(E_{m}',-)|_{\E}} \arrow[d, "{\T(b,-)|_{\E}}"] \arrow[lld, "{\T(\delta',-)|_{\E}}"'] \\
			& {{\Hom}_{\T}(\widetilde{E},-)|_{\E}} \arrow[r, "{\mathcal{T}(\widetilde{\alpha},-)|_{\E}}"'] \arrow[rdd, "{\mathcal{T}(\gamma,-)|_{\E}}"'] & {{\Hom}_{\T}(\widetilde{F},-)|_{\E}}                   & {{\Hom}_{\T}(F_{2},-)|_{\E}} \arrow[l, "{\mathcal{T}(\beta,-)|_{\E}}"] \arrow[dd, "\varphi_{2}"] \\
			&                                                                                                                                          &                                                      &                                                                                                \\
			{{\Hom}_{\T}(E',-)|_{\E}} \arrow[rr, "{\mathcal{T}(\alpha,-)|_{\E}}"'] \arrow[ruu, "{\mathcal{T}(\delta,-)|_{\E}}"] &                                                                                                                                          & {{\Hom}_{\T}(F_{1},-)|_{\E}} \arrow[r, "\varphi_{1}"'] & \mathbf{H}.                                                                                   
		\end{tikzcd}\]By Lemma \ref{extend weak triangle}(2), we can deduce that there is an exact sequence \[\mathbf{H}^{i-1}(E_{m}')\xrightarrow{\mathbf{H}^{i}(\sigma)}\mathbf{H}^{i}(\widetilde{F})\xrightarrow{\scalebox{0.6}{$\begin{pmatrix} \mathbf{H}^{i}(\widetilde{\alpha}) \\ \mathbf{H}^{i}(\beta) \end{pmatrix}$}} \mathbf{H}^{i}(\widetilde{E})\oplus \mathbf{H}^{i}(F_{2})\xrightarrow{(-\mathbf{H}^{i}(\delta'),\mathbf{H}^{i}(b))} \mathbf{H}^{i}(E_{m}')\]for all $i\in\mathbb{Z}$.
		Recall that there is a triangle $$D_{m}\stackrel{a}\longrightarrow  F_{2}\stackrel{b}\longrightarrow E_{m}'\longrightarrow D_{m}[1]$$
		with $D_{m}\in\T^{\geq m}$. Then $\mathbf{H}^{i}(b)$ is an isomorphism for $i \leq m-2$, 
		which implies $\mathbf{H}^{i}(\sigma) = 0$ for $i \leq m-1$. Consequently, we have a short exact sequence\[0\to \mathbf{H}^{i}(\widetilde{F})\xrightarrow{\scalebox{0.6}{$\begin{pmatrix} \mathbf{H}^{i}(\widetilde{\alpha}) \\ \mathbf{H}^{i}(\beta) \end{pmatrix}$}} \mathbf{H}^{i}(\widetilde{E})\oplus \mathbf{H}^{i}(F_{2})\xrightarrow{(-\mathbf{H}^{i}(\delta'),\mathbf{H}^{i}(b))} \mathbf{H}^{i}(E_{m}')\to 0\]for $i \leq m-2$. This yields a commutative diagram\[\begin{tikzcd}
			\mathbf{H}^{i}(\widetilde{F}) \arrow[r, "\mathbf{H}^{i}(\widetilde{\alpha})"] \arrow[d, "\mathbf{H}^{i}(\beta)"'] & \mathbf{H}^{i}(\widetilde{E}) \arrow[d, "\mathbf{H}^{i}(\delta')"] \\
			\mathbf{H}^{i}(F_{2}) \arrow[r, "\mathbf{H}^{i}(b)"']                                                    & \mathbf{H}^{i}(E_{m}')                                   
		\end{tikzcd}\]which is a pullback for $i \leq m-2$. Since $\mathbf{H}^{i}(b)$ is an isomorphism for $i\leq m-2$, 
		it follows that $\mathbf{H}^{i}(\widetilde{\alpha}): \mathbf{H}^{i}(\widetilde{F}) \to \mathbf{H}^{i}(\widetilde{E})$ 
		is also an isomorphism for $i \leq m-2$. As $\widetilde{E} \in \E \cap \T^{\geq 0}$, we have $\mathbf{H}^{i}(\widetilde{F}) = 0$ for $i < 0$. 
		By Lemma~\ref{tech1}, this implies $\widetilde{F} \in \T_{c}^{+} \cap \T^{\geq 0}$.
		
		{\bf Step 3:} Construct $\widetilde{\varphi}:{\Hom}_\T(\widetilde{F},-)|_{\E}\to\mathbf{H}$.
		
		By $(**)$, the composition \[{\Hom}_\T(E_{m}',-)|_{\E}\xrightarrow{\T((-\delta',b),-)|_{\E}}{\Hom}_\T(\widetilde{E}\oplus F_{2},-)|_{\E}\xrightarrow{(\varphi_{1}\circ\T(\ga,-)|_{\E},\varphi_{2})}\mathbf{H}\]is zero. Then by \cite[Lemma~8.5]{N18a}, we obtain a morphism $\widetilde{\varphi}:{\Hom}_\T(\widetilde{F},-)|_{\E}\to\mathbf{H}$ and a commutative diagram
		\[\begin{tikzcd}
			{{\Hom}_\T(\widetilde{E}\oplus F_{2},-)|_{\E}} \arrow[rr, "{\T(\scalebox{0.6}{$\begin{pmatrix} \widetilde{\alpha} \\ \beta \end{pmatrix}$},-)|_{\E}}"] \arrow[rrr, "{(\varphi_{1}\circ\T(\ga,-)|_{\E},\varphi_{2})}"', bend right] &  & {{\Hom}_\T(\widetilde{F},-)|_{\E}} \arrow[r, "\widetilde{\varphi}"] & \mathbf{H}.
		\end{tikzcd}\]
		This completes the proof.
	\end{proof}
	
	\begin{lem}\label{B.3}
		Let $\mathbf{H}$ be a locally finite $\E$-homological functor. Then there exists an integer $A > 0$ such that for every $i > 0$, the following hold.
		
		$(1)$ There exists a commutative diagram in $\mathcal{T}_{c}^{+} \cap \mathcal{T}^{\geq -A}$:
		\[\begin{tikzcd}
			E_{i+1} \arrow[rr, "\delta_{i}"] &                                                            & E_{i} \\
			& L_{i} \arrow[ru, "\alpha_{i}"'] \arrow[lu, "\gamma_{i+1}"] &      
		\end{tikzcd}\]
		where $E_{i}\in\E$, and a morphism $\eta_{i} : {\Hom}_\mathcal{T}(L_{i}, -)|_{\mathcal{E}} \to \mathbf{H}$ such that $\eta_{i}|_{\langle E \rangle_{i+2}}$ is an epimorphism. Moreover, $H^{\ell}(\gamma_{i})$ and $H^{\ell}(\alpha_{i})$ are isomorphisms for all $\ell \leq i$.
		
		$(2)$ For all $i > 0$, the following diagram commutes:
		\[\begin{tikzcd}
			{\Hom}_\mathcal{T}(E_{i+1}, -)|_{\mathcal{E}} \arrow[rr, "{\mathcal{T}(\alpha_{i+1}, -)|_{\mathcal{E}}}"] \arrow[d, "{\mathcal{T}(\gamma_{i+1}, -)|_{\mathcal{E}}}"'] & &{\Hom}_\mathcal{T}(L_{i+1}, -)|_{\mathcal{E}} \arrow[d, "\eta_{i+1}"] \\
			{\Hom}_\mathcal{T}(L_{i}, -)|_{\mathcal{E}} \arrow[rr, "\eta_{i}"'] & & \mathbf{H}.
		\end{tikzcd}\]
	\end{lem}
\begin{proof}
	  By Lemma~\ref{B.1}, there exists an integer \( A>0 \) such that for every \( i>0 \),  
	there are objects \( K_{i} \in \T_{c}^{+} \cap \T^{\ge -A} \) and morphisms  
	\( \psi_{i} \colon {\Hom}_{\T}(K_{i}, -)|_{\E} \to \mathbf{H} \)  
	with \( \psi_{i}|_{\langle E\rangle_{i}} \) an epimorphism.
	
	Set \( L_{1} = K_{3} \) and \( \eta_{1} = \psi_{3} \).  
	There is a triangle
 \[
	D_{1} \longrightarrow L_{1} \stackrel{\alpha_{1}}{\longrightarrow} E_{1} \longrightarrow D_{1}[1]
	\]
with \( D_{1} \in \T^{\ge 3} \) and \( E_{1} \in \E \).  
	Then \( \mathbf{H}^{i}(\alpha_{1}) \) is an isomorphism for all \( i \le 1 \).  
	Moreover, since \( \mathbf{H}^{i}(L_{1}) = 0 \) for all \( i \le -A-1 \),  
	Lemma~\ref{tech1} implies  
	\( E_{1} \in \E \cap \T^{\ge -A} \).
	
	By Lemma~\ref{B.2}(1), we obtain a commutative diagram
	\[
\begin{tikzcd}
& {\Hom_{\T}(E_{2},-)|_{\E}} \arrow[rdd, "{\T(\gamma_{2},-)|_{\E}}"'] \arrow[r, "{\T(\alpha_{2},-)|_{\E}}"] & {\Hom_{\T}(L_{2},-)|_{\E}} \arrow[rdd, "\eta_{2}"'] & {\Hom_{\T}(K_{4},-)|_{\E}} \arrow[dd, "\psi_{4}"] \arrow[l] \\  &   &   &               \\
{\Hom_{\T}(E_{1},-)|_{\E}} \arrow[rr, "{\T(\alpha_{1},-)|_{\E}}"'] \arrow[ruu, "{\T(\delta_{1},-)|_{\E}}"] &                                                                                                           & {\Hom_{\T}(L_{1},-)|_{\E}} \arrow[r, "\eta_{1}"']   & \mathbf{H}                                                 
\end{tikzcd}
	\]
with \( E_{2} \in \E \cap \T^{\ge -A} \) and  
	\( L_{2} \in \T_{c}^{+} \cap \T^{\ge -A} \).  
	By Lemma~\ref{B.2}(2), both  
	\( \mathbf{H}^{i}(\alpha_{2}) \) and \( \mathbf{H}^{i}(\gamma_{2}) \)  
	are isomorphisms for all \( i \le 2 \).  
	Since \( \psi_{4}|_{\langle E\rangle_{4}} \) is an epimorphism,  
	so is \( \eta_{2}|_{\langle E\rangle_{4}} \).
	
	Because \( E_{1}, E_{2} \in \E \), we have the commutative diagram
	\[
	\begin{tikzcd}
		& E_{2} \arrow[rr, "\delta_{1}"] & & E_{1} \\
		L_{2} \arrow[ru, "\alpha_{2}"] 
		& & 
		L_{1} \arrow[ru, "\alpha_{1}"'] \arrow[lu, "\gamma_{2}"] &
	\end{tikzcd}
	\]
	The general case now follows by induction.
\end{proof}
	\begin{prop}\label{epi1}
		Let $\mathbf{H}$ be a locally finite $\E$-homological functor. Then there exist an object $F \in \T_{c}^{+}$ and an epimorphism 
		$$\varphi:{\Hom}_\T(F, -)|_{\mathcal{\E}} \longrightarrow \mathbf{H}.$$
	\end{prop}
	
\begin{proof}
	By Lemma~\ref{B.3}(1), there exist an integer \( A>0 \) and a commutative diagram in  
	\( \T_{c}^{+} \cap \T^{\ge -A} \):
	\[
	\begin{tikzcd}
		\cdots \arrow[r] 
		& E_{3} \arrow[rr, "\delta_{2}"] 
		& & E_{2} \arrow[rr, "\delta_{1}"] 
		& & E_{1} \\
		& & L_{2} \arrow[ru, "\alpha_{2}"'] \arrow[lu, "\gamma_{3}"] 
		& & L_{1} \arrow[ru, "\alpha_{1}"'] \arrow[lu, "\gamma_{2}"] &
	\end{tikzcd}
	\]
where each \( E_{i} \in \E \), and  
	\( \mathbf{H}^{\ell}(\gamma_{i}) \) and \( \mathbf{H}^{\ell}(\alpha_{i}) \)  
	are isomorphisms for all \( \ell \le i \).  
	Moreover, for each \( i>0 \) there exists a morphism  
	\( \eta_{i} \colon {\Hom}_{\T}(L_{i}, -)|_{\E} \to \mathbf{H} \)  
	such that \( \eta_{i}|_{\langle E\rangle_{i+2}} \) is an epimorphism.  
	It follows that \( \mathbf{H}^{\ell}(\delta_{i}) \) is an isomorphism for all \( \ell \le i \).  
	Thus \( (E_{*}, \delta_{*}) \) is a strong \( \E \)-coapproximating system.
	
	Set \( F = \Holim(E_{*}) \).  
	By Lemma~\ref{universal property of coapproximating system}(1),  
	\( F \in \T_{c}^{+} \) and \( (E_{*}, \delta_{*}) \) is a strong  
	\( \E \)-coapproximating system for \( F \).  
	Lemma~\ref{colim-holim1} then yields a morphism
   \[
	\varphi \colon \colim \Hom_{\T}(E_{*}, -) \longrightarrow {\Hom}_{\T}(F, -)
	\]
whose restriction to \( \E \) is an isomorphism.
	
	For each \( i>0 \), Lemma~\ref{B.3} provides the commutative diagram
	\[
	\begin{tikzcd}
		{\Hom_{\T}(E_{i},-)|_{\E}} 
		\arrow[rr, "{\T(\delta_{i},-)|_{\E}}"] 
		\arrow[rrd, "{\T(\alpha_{i},-)|_{\E}}"'] 
		& & 
		{\Hom_{\T}(E_{i+1},-)|_{\E}} 
		\arrow[rr, "{\T(\alpha_{i+1},-)|_{\E}}"] 
		\arrow[d, "{\T(\gamma_{i+1},-)|_{\E}}"'] 
		& & 
		{\Hom_{\T}(L_{i+1},-)|_{\E}} 
		\arrow[d, "\eta_{i+1}"] \\
		& & 
		{\Hom_{\T}(L_{i},-)|_{\E}} 
		\arrow[rr, "\eta_{i}"'] 
		& & 
		\mathbf{H}.
	\end{tikzcd}
	\]
Hence there is a canonical morphism
 \[
	\eta \colon \colim \Hom_{\T}(E_{*}, -)|_{\E} \longrightarrow \mathbf{H},
	\]
induced by the compositions
\[
	{\Hom}_{\T}(E_{i}, -)|_{\E}
	\xrightarrow{\T(\alpha_{i}, -)|_{\E}}
	{\Hom}_{\T}(L_{i}, -)|_{\E}
	\xrightarrow{\eta_{i}}
	\mathbf{H}.
	\]	
	It remains to show that \( \eta \) is an epimorphism.  
	Let \( E' \in \E \).  
	Then there exists \( m>0 \) with \( E' \in \langle E\rangle_{m+2} \),  
	so \( \eta_{m}(E') \colon {\Hom}_{\T}(L_{m}, E') \to \mathbf{H}(E') \) is an epimorphism.  
	By Proposition~\ref{tech4}, there exists \( m'>0 \) such that  
	\( {\Hom}_{\T}(\T^{\ge m'}, E') = 0 \).  
	We may assume \( m = m' \).
	
	Extend \( \alpha_{m} \colon L_{m} \to E_{m} \) to a triangle
	\[
	D_{m} \longrightarrow L_{m} \stackrel{\alpha_{m}}{\longrightarrow} E_{m} \longrightarrow D_{m}[1].
	\]
Since \( \mathbf{H}^{i}(\alpha_{m}) \) is an isomorphism for all \( i \le m \),  
	we have \( \mathbf{H}^{i}(D_{m}) = 0 \) for all \( i \le m \).  
	Lemma~\ref{tech1} then implies \( D_{m} \in \T^{\ge m+1} \).  
	Thus
	\[
	\T(\alpha_{m}, E') \colon {\Hom}_{\T}(E_{m}, E') \longrightarrow {\Hom}_{\T}(L_{m}, E')
	\]
is an epimorphism.  
	Therefore the composition
	\[
	{\Hom}_{\T}(E_{m}, E')|_{\E}
	\xrightarrow{\T(\alpha_{m}, E')|_{\E}}
	{\Hom}_{\T}(L_{m}, E')|_{\E}
	\xrightarrow{\eta_{m}(E')}
	\mathbf{H}(E')
	\]
is an epimorphism.  
	Hence \( \eta \) is an epimorphism, completing the proof.
\end{proof}

Now we recall the definition of a \emph{phantom morphism}. This concept originated in topology \cite{AG64} and was introduced to triangulated categories by Neeman \cite[Definition 2.4]{N92b}. Let $\S$ be a triangulated category with coproducts. A morphism $f: M \to N$ is {\emph{phantom}} if for all $C \in\S^{c}$, the induced map $\S(C, f): {\Hom}_\S(C, M) \to {\Hom}_\S(C, N)$ is zero. We denote by $\mathcal{J}$ the class of all phantom morphisms. Furthermore, if $\S$ is compactly generated, $k$-linear and locally Hom-finite, Lemma~\ref{lem-ebasic}(2) implies that $f \in \mathcal{J}$ if and only if $\S(f, D) = 0$ for all $D \in \E$.

	\begin{lem}\label{cophantom decom}
		Let $F \in \T$ and suppose that ${\Hom}_\T(F,-)|_{\E}$ is a locally finite $\E$-homological functor. For any $n > 0$, there exists a triangle 
		$$D_n \stackrel{g}\longrightarrow F \stackrel{f}\longrightarrow F_n \longrightarrow D_n[1]$$ 
		with $F_n \in \T_c^+$, $g \in \mathcal{J}^n$, and $\T(f,-)|_{\E}$ is surjective.
	\end{lem}
	\begin{proof}
		We proceed by induction, beginning with the case $ n = 1 $. By Proposition~\ref{epi1}, there exist $F_1 \in \T_c^+$ and an epimorphism $\varphi_1 \colon {\Hom}_\T(F_1, -)|_{\mathcal{E}} \to {\Hom}_\T(F, -)|_{\mathcal{E}}$. Since $F_1 \in \T_c^+$, it follows from Proposition~\ref{Tc+coapproximable} and \cite[Lemma~7.8]{N18a} that there exists a morphism $f_1 \colon F \to F_1$ such that $\varphi_1 = \T(f_1, -)|_{\mathcal{E}}$. Extend $f_1$ to a triangle
		\[
		D_1 \stackrel{g_1}\longrightarrow F \stackrel{f_1}\longrightarrow F_1 \longrightarrow D_1[1].
		\]
		For any $C \in \mathcal{E}$, we obtain an exact sequence
		\[
		{\Hom}_\T(F_1, C) \xrightarrow{\T(f_1, C)} {\Hom}_\T(F, C) \xrightarrow{\T(g_1, C)} {\Hom}_\T(D_1, C).
		\]
		Since $\T(f_1, C)$ is surjective, we have $\T(g_1, C) = 0$, hence $g_1$ is a phantom morphism. This completes the case $ n = 1 $.
		
		Now assume the statement holds for all $n \leq k$ with $k \geq 1$, and consider the case $n = k + 1$. By the induction hypothesis, there exists a triangle
		\[
		D_k \xrightarrow{g_k} F \xrightarrow{f_k} F_k \to D_k[1],
		\]
		with $F_k \in \T_c^+$, $g_k \in \mathcal{J}^k$, and $\T(f_{k},-)|_{\E}$ is surjective. 
		Thus the sequence
		\[
		0 \longrightarrow {\Hom}_\T(D_k[1], -)|_{\mathcal{E}} \longrightarrow {\Hom}_\T(F_k, -)|_{\mathcal{E}} \xrightarrow{\T(f_k, -)|_{\mathcal{E}}} {\Hom}_\T(F, -)|_{\mathcal{E}}\longrightarrow 0
		\]
		is exact. Since both ${\Hom}_\T(F_k, -)|_{\mathcal{E}}$ and ${\Hom}_\T(F, -)|_{\mathcal{E}}$ are locally finite $\mathcal{E}$-homological functors, Lemma~\ref{ker of homo functor} implies that ${\Hom}_\T(D_k[1], -)|_{\mathcal{E}}$ is also a locally finite $\mathcal{E}$-homological functor. Applying the $n = 1$ case to this functor, we obtain a triangle
		\[
		D_{k+1} \stackrel{h}\longrightarrow D_k \longrightarrow F' \longrightarrow D_{k+1}[1],
		\]
		with $h \in \mathcal{J}$ and $F' \in \T_c^+$. By the octahedral axiom, we have the following commutative diagram
		\[\begin{tikzcd}
			D_{k+1} \arrow[r, "h"] \arrow[d, no head, shift right] \arrow[d, no head] & D_{k} \arrow[r] \arrow[d, "g_{k}"]                                & F' \arrow[r] \arrow[d]      & {D_{k+1}[1]} \arrow[d, no head, shift right] \arrow[d, no head] \\
			D_{k+1} \arrow[r, "g_{k}h"']                                              & F \arrow[d, "f_{k}"] \arrow[r, "f_{k+1}"']                        & F_{k+1} \arrow[d] \arrow[r] & {D_{k+1}[1]}                                                    \\                     & F_{k} \arrow[d] \arrow[r, no head] \arrow[r, no head, shift left] & F_{k} \arrow[d]             &                                                                 \\
			& {D_{k}[1]} \arrow[r]                                              & {F'[1]}                     &       
		\end{tikzcd}\]
		where all rows and columns are triangles. Clearly, $g_{k}h\in\mathcal{J}^{k+1}$ and $F_{k+1}\in\T_{c}^{+}$. Since $\T(f_{k},-)|_{\E}$ is surjective, the same holds for  $\T(f_{k+1},-)|_{\E}$. This completes the induction step and the proof.
	\end{proof}
	
	While the following lemma is well-known, we provide a proof for the reader's convenience.
	
	\begin{lem}\label{cophantom vanish}
		Let $n > 0$ and $C \in  \overline{[\mathcal{E}]}_{n}$. For any $ f: X \to Y $ in $ \mathcal{J}^{n} $, the induced map $\T(f,C): {\Hom}_\T(Y,C) \to {\Hom}_\T(X,C) $ vanishes.
	\end{lem}
\begin{proof}
	We proceed by induction, beginning with the case \( n = 1 \).  
	In this case \( f \in \mathcal{J} \), and there exists \( C' \in \T \) such that  
	\( C \oplus C' \simeq \prod_{i \in I} X_{i} \) with each \( X_{i} \in \E \).  
	Since
	\[
	\Hom_\T(f, \prod_{i \in I} X_{i}) \simeq \prod_{i \in I} \Hom_\T(f, X_{i}) = 0,
	\]
we obtain \( \Hom_\T(f, C) = 0 \).  
	This proves the case \( n = 1 \).
	
	Assume now that the statement holds for all \( n \le k \) with \( k \ge 1 \),  
	and consider the case \( n = k+1 \).  
	Then there exist an object \( Z \in \T \), a morphism \( g \colon X \to Z \) in \( \mathcal{J} \),  
	and a morphism \( h \colon Z \to Y \) in \( \mathcal{J}^{k} \) such that \( f = hg \).
	
	Let \( C \in \overline{[\E]}_{k+1} \).  
	Then there exists a triangle
	\[
	C_{1} \stackrel{s}{\longrightarrow} C \stackrel{t}{\longrightarrow} C_{2} \longrightarrow C_{1}[1]
	\]
with \( C_{1} \in \overline{[\E]}_{1} \) and \( C_{2} \in \overline{[\E]}_{k} \).  
	For any morphism \( \alpha \colon Y \to C \), consider the diagram
	\[
	\begin{tikzcd}
		& & & C_{1} \arrow[d, "s"] \\
		X \arrow[r, "g"] 
		& Z \arrow[r, "h"] 
		& Y \arrow[r, "\alpha"] 
		& C \arrow[d, "t"] \\
		& & & C_{2} \arrow[d] \\
		& & & C_{1}[1]
	\end{tikzcd}
	\]
By the induction hypothesis, \( t \alpha h = 0 \),  
	so \( \alpha h \) factors through \( C_{1} \).  
	Hence \( \alpha h g = \alpha f = 0 \).  
	This completes the proof.
\end{proof}

	\begin{thm}\label{thm:rep for T+c}
		Let $\T$ be a locally Hom-finite $k$-linear triangulated category with a compact silting object $G$. Consider the restricted Yoneda functor 
		$$y: {(\mathcal{T}_c^+)}^{\op} \longrightarrow \mathsf{Hom}_k(\mathcal{E}, k\text{-}\mathsf{Mod}),\quad T\longmapsto \Hom_{\T}(T,-)|_{\E}.$$
		Then $y$ is full, and the essential image consists of all the locally finite $\E$-homological functors. 
	\end{thm}
	
\begin{proof}
	By Lemma~\ref{reasonable}, the functor \( {\Hom}_{\T}(T,-)|_{\E} \) is locally finite whenever  
	\( T \in \T_{c}^{+} \).  
	Combining Corollary~\ref{key1} with \cite[Lemma~7.8]{N18a}, we conclude that \( y \) is full.
	
	Now let \( \mathbf{H} \) be a locally finite \( \E \)-homological functor.  
	It suffices to show that \( \mathbf{H} \in \mathrm{Im}\, y \).  
	By Proposition~\ref{brutal rep}, there exist \( F \in \T \) and an isomorphism
	\[
	\varphi \colon {\Hom}_{\T}(F,-)|_{\E} \longrightarrow \mathbf{H}.
	\]
	Moreover, by the proof of Lemma~\ref{tech5} together with Proposition~\ref{brutal rep},  
	we have \( F \in \overline{[\E]}_{4} \).  
	Applying Lemma~\ref{cophantom decom}, there exists a triangle
	\[
	D_{4} \stackrel{g}{\longrightarrow} F \stackrel{f}{\longrightarrow} F_{4} \longrightarrow D_{4}[1]
	\]
	with \( F_{4} \in \T_{c}^{+} \) and \( g \in \mathcal{J}^{4} \).  
	By Lemma~\ref{cophantom vanish}, the map
	\[
     \Hom_\T(g, F) \colon {\Hom}_{\T}(F, F) \longrightarrow {\Hom}_{\T}(D_{4}, F)
	\]
vanishes.  
	Hence \( g = 0 \), and therefore \( F \) is a direct summand of \( F_{4} \).  
	Since \( F_{4} \in \T_{c}^{+} \), it follows that \( F \in \T_{c}^{+} \) as well.
\end{proof}
	
	\begin{cor}\label{cor:rep for T+c}
		Let $X$ be an object in $\T$.
		Then $X\in\T^+_c$ if and only if the functor ${\Hom}_\T(X, -)|_{\E}$ is locally finite $\E$-homological.
	\end{cor}
\begin{proof}
	By Theorem~\ref{thm:rep for T+c}, there exist \( X' \in \T_{c}^{+} \) and an isomorphism
	\[
	\varphi \colon {\Hom}_{\T}(X', -)|_{\E} \longrightarrow {\Hom}_{\T}(X, -)|_{\E}.
	\]
	Since \( X' \in \T_{c}^{+} \), Corollary~\ref{key1} together with \cite[Lemma~7.8]{N18a} implies that there exists a morphism  
	\( f \colon X \to X' \) such that  
	\[
	\varphi = \T(f, -)|_{\E}.
	\]
	Extend \( f \) to a triangle
	\[
	X \stackrel{f}{\longrightarrow} X' \longrightarrow Y \longrightarrow X[1].
	\]
	For every \( i \in \mathbb{Z} \), we have \( {\Hom}_{\T}(Y, E[i]) = 0 \).  
	Since \( E \in \E \) is a cogenerator of \( \T \), it follows that \( Y = 0 \).  
	Hence \( X \simeq X' \in \T_{c}^{+} \).
\end{proof}

	\begin{lem}{\label{relation}}
		The following facts are given.
		\begin{enumerate}
			\item $\T^b\cap \T^-_c\subseteq \T^+_c$;
			\item $\T^b\cap \T^+_c\subseteq \T^-_c$;
			\item $\T^b_c=\T^b\cap \T^+_c=\T^+_c\cap \T^-_c$.
		\end{enumerate}
	\end{lem}
\begin{proof}
	Let \( X \in \T^{b} \cap \T^{-}_{c} \).  
	To show that \( X \in \T^{+}_{c} \), it suffices to prove that  
	\( {\Hom}_{\T}(X,-) \) is a locally finite \( \E \)-homological functor  
	(Corollary~\ref{cor:rep for T+c}).  
	Fix \( E' \in \E \).  
	Choose a triangle
	\[
	C \stackrel{\alpha}{\longrightarrow} X \stackrel{\beta}{\longrightarrow} D \longrightarrow C[1]
	\]
	with \( C \in \T^{c} \) and \( D \in \T^{\le -m} \).  
	Applying \( {\Hom}_{\T}(-,E') \) yields an exact sequence
	\[
	{\Hom}_{\T}(D,E') \longrightarrow {\Hom}_{\T}(X,E') \longrightarrow {\Hom}_{\T}(C,E')\longrightarrow{\Hom}_{\T}(D[-1],E').
	\]
	By Proposition~\ref{tech4}(1), we have \( E' \in \T^{+} \), hence  
	\[ {\Hom}_{\T}(D,E') = 0={\Hom}_{\T}(D[-1],E')\;\;\text{for}\;\;m \gg 0. \]  
	Since \( {\Hom}_{\T}(C,E') \) is finite-dimensional by Lemma~\ref{homfinitelem}(1),  
	it follows that \( {\Hom}_{\T}(X,E') \) is finite-dimensional.
	
	To prove \( {\Hom}_{\T}(X, E'[i]) = 0 \) for $i\gg 0$,  
	it suffices to show \( {\Hom}_{\T}(X, E[i]) = 0 \) for $i\gg0$ (Remark~\ref{reduce}). Since \( X \in \T^{b}\subseteq\T^{+} \), we have \( {\Hom}_{\T}(X, E[i]) = 0 \) for $i\gg0$ by Proposition~\ref{tech4}(1).
	
	\smallskip
	(2) is proved in the same way as (1).
	
	\smallskip
	(3) follows immediately from (1) and (2).
\end{proof}

	\begin{lem}\label{final_lem}
	Let $f \colon F \to F'$ be a morphism in $\T_{c}^{+}$ with $F \in \T_{c}^{b}$. Then $\T(f,-)|_{\E} = 0$ if and only if $f = 0$.
	\end{lem}
\begin{proof}
	The necessity is clear. For sufficiency, since \( F \in \T_{c}^{b} \), there exists an integer  
	\( \ell > 0 \) such that \( F \in \T^{\le \ell} \).  
	As \( F' \in \T_{c}^{+} \), we may choose a triangle
	\[
	D' \longrightarrow F' \stackrel{\alpha}{\longrightarrow} E' \longrightarrow D'[1]
	\]
  with \( D' \in \T^{\ge \ell+1} \) and \( E' \in \E \).  
	The condition \( \T(f,-)|_{\E} = 0 \) implies \( \alpha f = 0 \), so \( f \) factors through \( D' \).  
	Since \( {\Hom}_{\T}(F, D') = 0 \), it follows that \( f = 0 \).  
\end{proof}

	\begin{thm}\label{thm: rep for Tbc}
		Let $\T$ be a locally Hom-finite $k$-linear triangulated category with a compact silting object $G$.
		Consider the composite
		$${(\T_c^b)}^{\op}\stackrel{i}\longrightarrow{(\T_c^+)}^{\op}\stackrel{y}\longrightarrow {\Hom}_{k}(\E,k{\mbox -}{\rm Mod}).$$
		Then the composite functor $y\circ i$ is fully faithful, and the essential image consists of all the finite $\E$-homological functors.
	\end{thm}
	
\begin{proof}
	By Lemma~\ref{relation}(3), the functor  
	\( i \colon (\T_{c}^{b})^{\op} \to (\T_{c}^{+})^{\op} \)  
	is well-defined.  
	Combining Theorem~\ref{thm:rep for T+c} with Lemma~\ref{final_lem}, we conclude that  
	\( y \circ i \) is fully faithful.
	
	Let \( X \in \T_{c}^{b} = \T_{c}^{+} \cap \T^{b} \).  
	By Lemma~\ref{reasonable}(2), the functor  
	\( {\Hom}_{\T}(X,-)|_{\E} \) is finite.
	
	Now suppose \( Y \in \T \) is such that  
	\( {\Hom}_{\T}(Y,-)|_{\E} \) is a finite \( \E \)-homological functor.  
	It suffices to show that \( Y \in \T_{c}^{b} \).  
	By Corollary~\ref{cor:rep for T+c}, we already have \( Y \in \T_{c}^{+} \).  
	For any \( i \in \mathbb{Z} \),
	\[
	{\Hom}_{\T}(Y, E[i]) \simeq \mathbf{D}{\Hom}_{\T}(G[i], Y).
	\]
	Since \( {\Hom}_{\T}(Y,-)|_{\E} \) is finite,  
	\( {\Hom}_{\T}(G[i], Y) = 0 \) for \( |i| \gg 0 \). By \cite[Corollary~3.10]{BNP18}, \( Y \in \T^{b} \). Therefore  
	\( Y \in \T^{b} \cap \T_{c}^{+} = \T_{c}^{b} \) by Lemma~\ref{relation}(3).
\end{proof}

	The proof of the above theorem essentially yields the following corollary.
	
	\begin{cor}
	    Let $X$ be an object in $\T$.
		Then $X\in\T^b_c$ if and only if the functor ${\Hom}_\T(X, -)|_{\E}$ is finite $\E$-homological.
	\end{cor}

\section{Localization theorems for $\T^+_c$ and $\E$}\label{sec:loc thm}
	In this section, we establish localization theorems for the intrinsic subcategories $\T^+_c$ and $\E$, extending the framework from \cite{SZZ} to the Brown--Comenetz duality setting. In the classical theory, recollements of triangulated categories induce short exact sequences of Verdier quotients, 
	and these sequences play a central role in understanding how homological and structural information decomposes across the recollement.  
	We show that analogous localization phenomena hold for the dual intrinsic subcategories 
	$\T^+_c$ and $\E$, which arise naturally from Brown--Comenetz duality and the representability theory developed in the previous section. These results constitute the dual counterparts of the localization theorems in \cite{SZZ}.

    Throughout this section, for a triangulated category $\mathcal{T}$ with a compact generator $G$, we denote $(\mathcal{T}_G^{\leq 0}, \mathcal{T}_G^{\geq 0})$ by $(\mathcal{T}^{\leq 0}, \mathcal{T}^{\geq 0})$.
	
	\subsection{A localization theorem for $\T^+_c$}
	
	\begin{lem}\label{lem:res to T+c}
    Let $\mathbf{F}:\T \to \S$ be a triangle functor between compactly generated $k$-linear triangulated categories, each of which has a compact generator. If $\mathbf{F}$ preserves objects of $\E_t$ and products, then $\mathbf{F}$ can restrict to $\mathbf{F}:\T^+_c \to \S^+_c$. 
	\end{lem}
\begin{proof}
	Since \( \T \) is a compactly generated triangulated category, the dual Brown representability theorem  
	(see \cite[Proposition~5.3.1]{K10}) ensures that \( \mathbf{F} \) admits a left adjoint  
	\( \mathbf{F}' \colon \S \to \T \).  
	 
	Let \( G_{s}\) and \( G_{t} \) be compact generators of $\S$ and $\T$ respectively, and let \( E_{t} \) denote the Brown--Comenetz dual of \( G_{t} \).  
	Then
	\[
	\mathbf{D}{\Hom}_{\T}(G_{t}[i], \mathbf{F}'(G_{s}))
	\simeq {\Hom}_{\T}(\mathbf{F}'(G_{s}), E_{t}[i])
	\qquad\text{(Brown--Comenetz dual)}
	\]
	\[\quad\quad\quad\quad\quad
	\simeq {\Hom}_{\S}(G_{s}, \mathbf{F}(E_{t})[i])
	\qquad\text{(adjunction)}.
	\]
Since \( \mathbf{F} \) respects objects in \( \E_{t} \), we have  
	\( \mathbf{F}(E_{t}) \in \E_{s} \subseteq \S^{+} \) by Proposition~\ref{tech4}(1).  
	Thus ${\Hom}_{\S}(G_{s}, \mathbf{F}(E_{t})[i]) = 0$ for $i \ll 0$. Hence, ${\Hom}_{\T}(G_{t}[i], \mathbf{F}'(G_{s})) = 0$ for $i \ll 0$.
	By \cite[Corollary~3.10]{BNP18},  
	\( \mathbf{F}'(G_{s}) \in \T^{\le n} \) for some \( n>0 \).  
	For any \( M \in \T^{\ge 1} \),
	\[
	{\Hom}_{\S}(G_{s}[i], \mathbf{F}(M))
	\simeq {\Hom}_{\T}(\mathbf{F}'(G_{s})[i], M)
	= 0 \quad \text{for } i \ge n.
	\]
	Then $\mathbf{F}(M)\in\S^{\geq -n+1}$ and hence \( \mathbf{F}(\T^{\ge 1}) \subseteq \S^{\ge -n+1} \).
	
	Now let \( l>0 \), set \( m = n+l \), and take \( T \in \T_{c}^{+} \).  
	Choose a triangle
	\[
	D \longrightarrow T \longrightarrow F \longrightarrow D[1]
	\]
	with \( D \in \T^{\ge m} \) and \( F \in \E_{t} \).  
	Applying \( \mathbf{F} \) gives a triangle
	\[
	\mathbf{F}(D) \longrightarrow \mathbf{F}(T) \longrightarrow \mathbf{F}(F) \longrightarrow \mathbf{F}(D)[1],
	\]
	where \( \mathbf{F}(D) \in \S^{\ge l} \) and \( \mathbf{F}(F) \in \E_{s} \).  
	Hence \( \mathbf{F}(T) \in \S_{c}^{+} \).
\end{proof}

	\begin{lem}\label{lem:2 res to T+c}
		Let $\mathbf{F}:\mathcal{T}\to\mathcal{S}$ be a triangle functor between locally Hom-finite $k$-linear triangulated categories, each of which has a compact silting object. If $\mathbf{F}$ admits a right adjoint functor that preserves objects of $\E_{s}$, then $\mathbf{F}$ restricts to $\T^+_c \to \S^+_c$. 
	\end{lem}
\begin{proof}
	Let \( X \in \T_{c}^{+} \) and \( Y \in \E_{s} \).  
	Assume that \( \mathbf{F} \) admits a right adjoint \( \mathbf{G} \).  
	Then for all \( i \in \mathbb{Z} \),
	\[
	{\Hom}_{\S}(\mathbf{F}(X), Y[i])
	\simeq{\Hom}_{\T}(X, \mathbf{G}(Y)[i]).
	\]
	Since \( \mathbf{G}(Y) \in \E_{t} \), Theorem~\ref{thm:rep for T+c} implies that  
	\( {\Hom}_{\T}(X, \mathbf{G}(E)) \) is finite-dimensional and
	\[
	{\Hom}_{\T}(X, \mathbf{G}(E)[i]) = 0 \quad \text{for } i \gg 0.
	\]
	Hence  
	\( {\Hom}_{\S}(\mathbf{F}(X), Y) \) is finite-dimensional and  
	\( {\Hom}_{\S}(\mathbf{F}(X), Y[i]) = 0 \) for \( i \gg 0 \).  
	Theorem~\ref{thm:rep for T+c} implies  \( \mathbf{F}(X) \in \S_{c}^{+} \).
\end{proof}

	\begin{lem}\label{lem:E in Tbc}
		Let $\T$ be a locally Hom-finite $k$-linear triangulated category with a compact silting object. 
		Suppose that $\T$ admits a compact generator $G$ satisfying ${\Hom}_{\T}(G, G[i])=0$ for $i\ll 0$. Let $E$ be the Brown--Comenetz dual of $G$. Then $G\in\T^b_c$ and $E\in \T^b_c$. Moreover, $G$ is in $\T^+_c$ and $E$ is in $\T^-_c$.
	\end{lem}
\begin{proof}
	The statement \( G \in \T^{b}_{c} \) follows from \cite[Lemma~2.10]{SZZ}.  
	The assertion \( E \in \T^{b}_{c} \) is an immediate consequence of  
	Theorem~\ref{lem:object in bc}(2). By Lemma~\ref{relation}, \(\T_{c}^{b}=\T_{c}^{+}\cap\T_{c}^{-}\), completing the proof.
\end{proof}

	\begin{lem}\label{lem:hom char Tc}
		Let $\T$ be a locally Hom-finite $k$-linear triangulated category with a compact silting object $G$, and let $E$ be the Brown--Comenetz dual of $G$. Let $X \in \T$. Then
        \begin{enumerate}
            \item $X \in \T_{c}^{-}$ if and only if ${\Hom}_{\T}(G, X[i])$ is finite-dimensional for all $i$ and vanishes for $i \gg 0$;    
            \item $X \in \T_{c}^{b}$ if and only if ${\Hom}_{\T}(G, X[i])$ is finite-dimensional for all $i$ and vanishes for $|i| \gg 0$;
            \item $X \in \T_{c}^{+}$ if and only if ${\Hom}_{\T}(X[i], E)$ is finite-dimensional for all $i$ and vanishes for $i \ll 0$, if and only if ${\Hom}_{\T}(G, X[i])$ is finite-dimensional for all $i$ and vanishes for $i \ll 0$.
        \end{enumerate}
	\end{lem}
\begin{proof}
	The characterizations of $\T_c^-$ and $\T_c^b$ follow from Theorem~\ref{lem:object in bc}. The characterization of $\T_c^+$ follows from the representability theorem in Section~\ref{sec:rep thm} and Brown--Comenetz duality.
\end{proof}

	\begin{lem}\label{lem: res t-struc}
		Let $\T$ be a locally Hom-finite $k$-linear triangulated category with a compact silting object $G$. Then $(\T^{\leq 0},\T^{\geq 0})$ can be restricted to $\T_c^-, \T_c^b$ and $\T_c^+$.
	\end{lem}
	
\begin{proof}
	We use the characterizations in Lemma~\ref{lem:hom char Tc}.
	We show that the \( t \)-structure restricts to \( \T_{c}^{-} \), the other cases are analogous.  
For any \( X \in \T_{c}^{-} \), consider its truncation triangle
\[
X^{\le 0} \longrightarrow X \longrightarrow X^{\ge 1} \longrightarrow X^{\le 0}[1].
\]
Since \( G \in \T^{\leq 0} \), we have \( \Hom_{\T}(G[i], X^{\leq 0}) \simeq \Hom_{\T}(G[i], X) \) for \( i \geq 0 \). Because \( G \) is silting, by definition we have
\[
\T^{\leq 0} = \{ M \in \T \mid \Hom_{\T}(G, M[i]) = 0 \text{ for } i > 0 \}.
\]
Consequently, \( \Hom_{\T}(G[i], X^{\leq 0}) = 0 \) for \( i < 0 \). This implies \( X^{\leq 0} \in \T_{c}^{-} \). Since \( \T_{c}^{-} \) is a triangulated subcategory of \( \T \), we also obtain \( X^{\geq 1} \in \T_{c}^{-} \). Thus \( (\T^{\leq 0}, \T^{\geq 0}) \) can be restricted to \( \T_{c}^{-} \).
\end{proof}

	By Lemma~\ref{lem: res t-struc}, locally Hom-finite triangulated categories with compact silting objects satisfy both the coherent condition and its dual in the sense of \cite[Definition~5.1]{N18c}.  
	This structural control is essential for establishing localization sequences for $\T^+_c$.
	
\begin{lem}
Let the following diagram be a recollement of triangulated categories, each of which admits a compact silting object.
	$$\xymatrix{\R\ar^-{i_*=i_!}[r]
			&\T\ar^-{j^!=j^*}[r]\ar^-{i^!}@/^1.2pc/[l]\ar_-{i^*}@/_1.6pc/[l]
			&\S.\ar^-{j_*}@/^1.2pc/[l]\ar_-{j_!}@/_1.6pc/[l]}$$
Then the functors $i^*,i_*,j_!,j^!$ are left $t$-bounded with respect to $t$-structures generated by silting objects, and the functors $i_*,i^!,j^*,j_*$ are right $t$-bounded.
\end{lem}

\begin{proof}
Let $(\R^{\leq 0},\R^{\geq 0})$,$(\T^{\leq 0},\T^{\geq 0})$ and $(\S^{\leq 0},\S^{\geq 0})$ are the $t$-structures generated by silting objects.
Thanks to Corollary 3.12 in \cite{BNP18}, the $t$-structure $(\T_{gl}^{\leq 0},\T_{gl}^{\geq 0})$ obtained by gluing $(\R^{\leq 0},\R^{\geq 0})$ and $(\S^{\leq 0},\S^{\geq 0})$ is equivalent to $(\T^{\leq 0},\T^{\geq 0})$. We show $j_!$ is left $t$-bounded and others can be obtained similarly. Since there exists $n\geq 0$ such that $\T^{\leq 0}\subset \T_{gl}^{\leq n}$, then we have $j_!(\S^{\leq 0})\subset \T_{gl}^{\leq 0}\subset \T^{\leq n}$ for some $n\geq 0$.
\end{proof}

\begin{thm}\label{thm:loc for T+c}
	Let the following diagram be a recollement of locally Hom-finite $k$-linear triangulated categories, each of which admits a compact silting object.
	$$\xymatrix{\R\ar^-{i_*=i_!}[r]
			&\T\ar^-{j^!=j^*}[r]\ar^-{i^!}@/^1.2pc/[l]\ar_-{i^*}@/_1.6pc/[l]
			&\S.\ar^-{j_*}@/^1.2pc/[l]\ar_-{j_!}@/_1.6pc/[l]}$$
Then 
		\begin{enumerate}
			\item Suppose that $\T$ has a compact generator $G_t$ such that there is an integer $N$ with ${\Hom}_{\T}(G_t, G_t[n])=0, n<N$. If the recollement can extend one step upwards, then the first row induces a short exact sequence $${\S}^+_c/{\S}^c\stackrel{\overline{j_!}}{\longrightarrow}\T^+_c/\T^c\stackrel{\overline{i^*}}{\longrightarrow}{\R}^+_c/{\R}^c.$$
			\item The second row induces a short exact sequence    $${\R}^+_c/{\R}^b_c\stackrel{\overline{i_*}}{\longrightarrow}\T^+_c/\T^b_c\stackrel{\overline{j^*}}{\longrightarrow}{\S}^+_c/{\S}^b_c.$$
		\end{enumerate}
\end{thm}
	
\begin{proof}
	(1) By Lemma~\ref{lem:E in Tbc}, we have \( \T^{c} \subseteq \T^{+}_{c} \).  
	From the proof of \cite[Corollary~3.3]{SZZ}, the compact generators in \( \R \) and \( \S \) also satisfy the assumptions of Lemma~\ref{lem:E in Tbc}, hence  
	\( \R^{c} \subseteq \R^{+}_{c} \) and \( \S^{c} \subseteq \S^{+}_{c} \).  
	By \cite[Lemma~2.5]{SZZ} and \cite[Lemma~2.6]{SZZ}, to obtain the exact sequence it suffices to verify the following two claims.

\textbf{Claim 1.}  
	Every morphism from \( \T^{c} \) to \( j_{!}(\S^{+}_{c}) \) factors through \( j_{!}(\S^{c}) \).
	
 \textbf{Claim 2.}  
	\( j_{!}(\S^{+}_{c}) \cap \T^{c} = j_{!}(\S^{c}) \).
	
	Let \( T \in \T^{c} \), \( V \in \S^{+}_{c} \), and \( f \colon T \to j_{!}(V) \).  
	Since \( j_{!} \) admits a left adjoint \( j^{\#} \), we have an adjunction isomorphism
	\[
	{\Hom}_{\T}(T, j_{!}(V)) \simeq {\Hom}_{\S}(j^{\#}(T), V).
	\]
	Thus \( f \) is the composite
	\[
	T \xrightarrow{\eta_{T}} j_{!}j^{\#}(T) \xrightarrow{j_{!}(g)} j_{!}(V)
	\]
	for some \( g \colon j^{\#}(T) \to V \).  
	Since \( j^{\#} \) preserves compact objects, \( j^{\#}(T) \in \S^{c} \), and hence  
	\( f \) factors through \( j_{!}j^{\#}(T) \in j_{!}(\S^{c}) \).  
	This proves Claim 1. Claim 2 follows from Claim 1.
	
	(2) By Proposition~\ref{prop:dual neeman}, Lemma~\ref{lem:res to T+c} and Lemma~\ref{lem:2 res to T+c}, we have a half recollement
	\[
	\xymatrix{
		\R^{+}_{c} \ar[r]^{i_{*}=i_{!}} 
		& \T^{+}_{c} \ar[r]^{j^{!}=j^{*}} \ar@/^1.2pc/[l]^{i^{!}} 
		& \S^{+}_{c} \ar@/^1.2pc/[l]^{j_{*}} }.
	\]
    Hence, we have the following exact sequence:
    $$\R^+_c\stackrel{i_*}\longrightarrow \T^+_c\stackrel{j^*}\longrightarrow \S^+_c.$$
    By \cite[Theorem~3.5]{SZZ}, we have the following short exact sequence up to direct summands:
    $$\R^b_c\stackrel{i_*}\longrightarrow \T^b_c\stackrel{j^*}\longrightarrow \S^b_c.$$
    
	By \cite[Lemma~2.5]{SZZ} and \cite[Lemma~2.6]{SZZ}, to obtain the short exact sequence of Verdier quotients it suffices to prove the following two claims.

    \textbf{Claim 1.}  
	Every morphism from \( \T^{b}_{c} \) to \( i_{*}(\R^{+}_{c}) \) factors through \( i_{*}(\R^{b}_{c}) \).
	
	\textbf{Claim 2.}  
	\( i_{*}(\R^{+}_{c}) \cap \T^{b}_{c} = i_{*}(\R^{b}_{c}) \).  
		
	Let \( X \in \T^{b}_{c} \), \( Y \in \R^{+}_{c} \), and \( f \colon X \to i_{*}(Y) \).  
	Since \( X \in \T^{b}_{c} \), there exists \( m \) with \( X \in \T^{\le m} \).  
	Because \( i_{*} \) is right \( t \)-bounded, there exists \( n \) such that  
	\( i_{*}(\R^{\ge 0}) \subseteq \T^{\ge n} \).
	
	By Lemma~\ref{lem: res t-struc}, for \( Y \in \R^{+}_{c} \) there is a triangle
	\[
	Y_{1} \longrightarrow Y \longrightarrow Y_{2} \longrightarrow Y_{1}[1]
	\]
	with  
	\( Y_{1} \in \R^{+}_{c} \cap \R^{\le m-n} \subseteq \R^{b}_{c} \)  
	and  
	\( Y_{2} \in \R^{\ge m-n+1} \).  
	Applying \( i_{*} \) gives a commutative diagram
	$$\xymatrix{
		& X \ar[d]^{f} \ar@{-->}[dl]_{g} & \\
		i_{*}(Y_{1}) \ar[r] & i_{*}(Y) \ar[r] & i_{*}(Y_{2}) \ar[r] & i_{*}(Y_{1})[1]
	}$$
	Since $i_*(Y_2)\in i_*(\R^{\geq m-n+1})\subset \T^{\geq m+1}$, we have \( {\Hom}_{\T}(X, i_{*}(Y_{2})) = 0 \), the morphism \( f \) factors through  
	$i_*(Y_1)\to i_*(Y)$.  Thus \( f \) factors through \( i_{*}(Y_{1}) \in i_{*}(\R^{b}_{c}) \), proving Claim 1. Claim 2 follows from Claim 1.
	
	\medskip
	This completes the proof.
\end{proof}

	\begin{cor}
		Let the following diagram be a recollement of locally Hom-finite $k$-linear triangulated categories, each of which admits a compact silting object. 
		$$\xymatrix{\R\ar^-{i_*=i_!}[r]
			&\T\ar^-{j^!=j^*}[r]\ar^-{i^!}@/^1.2pc/[l]\ar_-{i^*}@/_1.6pc/[l]
			&\S.\ar^-{j_*}@/^1.2pc/[l]\ar_-{j_!}@/_1.6pc/[l]}$$
		Suppose that $\T$ has a compact generator $G_t$ such that there is an integer $N$ with ${\Hom}_{\T}(G_t, G_t[n])=0, n<N$. If the recollement can extend one step downwards, then the recollement induces a commutative diagram
		$$\xymatrix{
			\R^b_c/\R^c\ar[r]^{\overline{i_*}}\ar@{^(->}[d]&\T^b_c/\T^c\ar[r]^{\overline{j^*}}\ar@{^(->}[d]&\S^b_c/\S^c\ar@{^(->}[d]
			\\
			\R^+_c/\R^c\ar[r]^{\overline{i_*}}\ar[d]&\T^+_c/\T^c\ar[r]^{\overline{j^*}}\ar[d]&\S^+_c/\S^c\ar[d]
			\\
			\R^+_c/\R^b_c\ar[r]^{\overline{i_*}}&\T^+_c/\T^b_c\ar[r]^{\overline{j^*}}&\S^+_c/\S^b_c}$$
		in which all rows and columns are short exact sequences of quotient categories.
	\end{cor}
	
    \begin{proof}
Under the assumption ${\Hom}_{\T}(G_t, G_t[n])=0, n<N$, by Lemma \ref{lem:E in Tbc}, we know $\T^c\subset \T_c^b, \R^c\subset \R_c^b$ and $\S^c\subset \S_c^b$. The exactness of the first row comes from \cite[Corollary~3.3]{SZZ}. And the exactness of the second row and third row come from Theorem \ref{thm:loc for T+c}.

    \end{proof}
	\subsection{A localization theorem for $\E$}
	
	\begin{thm}\label{thm:loc for E}
		Let the following diagram be a recollement of locally Hom-finite approximable $k$-linear triangulated categories 
		$$\xymatrix{\R\ar^-{i_*=i_!}[r]
			&\T\ar^-{j^!=j^*}[r]\ar^-{i^!}@/^1.2pc/[l]\ar_-{i^*}@/_1.6pc/[l]
			&\S.\ar^-{j_*}@/^1.2pc/[l]\ar_-{j_!}@/_1.6pc/[l]}$$ 
		\begin{enumerate}
			\item Let $\U,\V$ and $\W$ be triangulated subcategories of $\R, \S$ and $\T$, respectively, such that $\E_r\subseteq\U$, $\E_s\subseteq\V\subseteq {\S}^+_c$ and $\E_t\subseteq \W$. If the third row of the recollement is restricted to a short exact sequence
			$$\V\stackrel{j_*}{\longrightarrow}\W\stackrel{i^!}{\longrightarrow}\U,$$
			then it induces a short exact sequence
			$$\V/\E_s\stackrel{\overline{j_*}}{\longrightarrow}\W/\E_t\stackrel{\overline{i^!}}{\longrightarrow}\U/\E_r.$$
			\item Moreover, if the triangulated categories $\R,\T$ and $\S$ have compact silting objects, then the third row induces a short exact sequence
			$${\S}^+_c/\E_s\stackrel{\overline{j_*}}{\longrightarrow}\T^+_c/\E_t\stackrel{\overline{i^!}}{\longrightarrow}{\R}^+_c/\E_r.$$
		\end{enumerate}
	\end{thm}

	\begin{proof}
		(1) There is the following commutative diagram
		$$\xymatrix{
			\E_s\ar[r]^{j_*}\ar@{^(->}[d]&\E_t\ar[r]^{i^!}\ar@{^(->}[d]&\E_r\ar@{^(->}[d]
			\\
			\V\ar[r]^{j_*}\ar[d]&\W\ar[r]^{i^!}\ar[d]&\U\ar[d]
			\\
			\V/\E_s\ar[r]^{\overline{j_*}}&\W/\E_t\ar[r]^{\overline{i^!}}&\U/\E_r
		}$$
		in which the first row is exact up to direct summands by Proposition~\ref{prop:dual neeman} and the second row is exact by hypothesis. By \cite[Lemma~2.5]{SZZ}, it suffices to prove the induced functor ${\overline{j_*}}$ is fully-faithful.
		
		Note that ${\overline{j_*}}$ is really the composition
		$$\V/\E_s\stackrel{\tilde{j_*}}{\longrightarrow} j_*(\V)/j_*(\E_s)\stackrel{i}\longrightarrow \W/\E_t.$$
		Since $j_*$ is fully faithful, so $\tilde{j^*}$ in the decomposition above is an equivalence. It remains to prove that $i$ is fully-faithful. By  \cite[Lemma~2.6]{SZZ}, to obtain the short exact sequence, it remains to check the following two claims.

        {\bf Claim 1:} Each morphism from $j_*(\V)$ to $\E_t$ factors through an object in $j_*(\E_s)$.
		
		{\bf Claim 2:} $j_*(\V)\cap \E_t=j_*(\E_s)$. 
		
		Let $V$ be an object in $\V\subseteq \S^+_c$. $V$ admits a triangle
		$$D\stackrel{}\longrightarrow V\stackrel{}\longrightarrow F\stackrel{}\longrightarrow D[1]$$
		with $D\in\S^{\geq m}$ and $F\in \E_s$. Since $j_*$ is left $t$-bounded, for sufficiently large $m$, we have $\Hom_{\T}(j_*(D), E_t)=0$ by Proposition~\ref{tech4}. Hence each morphism from $j_*(V)$ to $E_t$ factors through $j_*(F)$, which belongs to $j_*(\E_s)$. This proves Claim 1. Claim 2 follows from Claim 1.
		
		(2) From the proof of Theorem \ref{thm:loc for T+c}(2) that there is an exact sequence
		$$\S^+_c\stackrel{j_*}\longrightarrow \T^+_c\stackrel{i^!}\longrightarrow \R^+_c.$$
		By (1), we obtain the short exact sequence of Verdier quotient categories.
	\end{proof}
	
	\begin{cor}
		Let the following diagram be a recollement of locally Hom-finite $k$-linear triangulated categories, each of which admits a compact silting object. 
		$$\xymatrix{\R\ar^-{i_*=i_!}[r]
			&\T\ar^-{j^!=j^*}[r]\ar^-{i^!}@/^1.2pc/[l]\ar_-{i^*}@/_1.6pc/[l]
			&\S.\ar^-{j_*}@/^1.2pc/[l]\ar_-{j_!}@/_1.6pc/[l]}$$ 
		If the recollement can extend one step upwards, then there exists a commutative diagram
		$$\xymatrix{
			\R^b_c/\E_r\ar[r]^{\overline{i_*}}\ar@{^(->}[d]&\T^b_c/\E_t\ar[r]^{\overline{j^!}}\ar@{^(->}[d]&\S^b_c/\E_s\ar@{^(->}[d]
			\\
			\R^+_c/\E_r\ar[r]^{\overline{i_*}}\ar[d]&\T^+_c/\E_t\ar[r]^{\overline{j^!}}\ar[d]&\S^+_c/\E_s\ar[d]
			\\
			\R^+_c/\R^b_c\ar[r]^{\overline{i_*}}&\T^+_c/\T^b_c\ar[r]^{\overline{j^!}}&\S^+_c/\S^b_c}$$
		in which all rows and columns are short exact sequences of quotient categories.
	\end{cor}
	
	\begin{exam}
	Let the following diagram be a recollement of derived categories of finite-dimensional algebras over a field $k$. 
	$$\xymatrix{\D(B)\ar^-{i_*=i_!}[r]
		&\D(A)\ar^-{j^!=j^*}[r]\ar^-{i^!}@/^1.2pc/[l]\ar_-{i^*}@/_1.6pc/[l]
		&\D(C).\ar^-{j_*}@/^1.2pc/[l]\ar_-{j_!}@/_1.6pc/[l]}$$
	The derived categories $\D(A),\D(B)$ and $\D(C)$ are locally Hom-finite $k$-linear triangulated categories with compact silting objects $A, B$ and $C$, respectively. 
	\begin{enumerate}
		\item Proposition~\ref{prop:dual neeman} implies that the third row of the recollement restricts to a short exact sequence up to direct summands
		$$\K^b(C\inj)\stackrel{j_*}\longrightarrow \K^b(A\inj)\stackrel{i^!}\longrightarrow \K^b(B\inj).$$
		\item Theorem~\ref{thm:loc for T+c} implies that we have short exact sequences
		$$\K^+(B\inj)/\D^b(B\smod)\stackrel{\overline{i_*}}\longrightarrow \K^+(A\inj)/\D^b(A\smod)\stackrel{\overline{j^*}}\longrightarrow \K^+(C\inj)/\D^b(C\smod).$$
		
		\item Theorem~\ref{thm:loc for E} gives us the following short exact sequence
		$$\K^+(C\inj)/\K^b(C\inj)\stackrel{\overline{j_*}}\longrightarrow \K^+(A\inj)/\K^b(A\inj)\stackrel{\overline{i^!}}\longrightarrow \K^+(B\inj)/\K^b(B\inj).$$
	\end{enumerate}
\end{exam}

{\bf Acknowledgments:} The second author was supported by the National Natural Science Foundation of China (Grant No.12501052). The third author was supported by the National Natural Science Foundation of China (Grant No.12401044) and the Hubei University Original Exploration Seed Fund (Grant No.260701747001).
    

\end{document}